\documentclass[openany, amssymb, psamsfonts]{amsart}
\usepackage{amssymb}
\usepackage{mathrsfs,comment}
\usepackage{biblatex}
\usepackage[normalem]{ulem}
\usepackage{setspace}
\usepackage{graphicx}
\usepackage[margin=1in]{geometry}
\usepackage{url}
\usepackage[all,arc,2cell]{xy}
\UseAllTwocells
\usepackage{enumerate}
\usepackage{hyperref}  
\hypersetup{%
  bookmarksnumbered=true,%
  bookmarks=true,%
  colorlinks=false,%
  linkcolor=blue,%
  citecolor=blue,%
  filecolor=blue,%
  menucolor=blue,%
  pagecolor=blue,%
  urlcolor=blue,%
  pdfborder={0 0 0},%
  pdfnewwindow=true,%
  pdfstartview=FitBH}   
\usepackage{my-macros}

\def\makeautorefname#1#2{\expandafter\def\csname#1autorefname\endcsname{#2}}
\def\equationautorefname~#1\null{(#1)\null}
\makeautorefname{footnote}{footnote}%
\makeautorefname{item}{item}%
\makeautorefname{figure}{Figure}%
\makeautorefname{table}{Table}%
\makeautorefname{part}{Part}%
\makeautorefname{appendix}{Appendix}%
\makeautorefname{chapter}{Chapter}%
\makeautorefname{section}{Section}%
\makeautorefname{subsection}{Section}%
\makeautorefname{subsubsection}{Section}%
\makeautorefname{theorem}{Theorem}%
\makeautorefname{corollary}{Corollary}%
\makeautorefname{lemma}{Lemma}%
\makeautorefname{proposition}{Proposition}%
\makeautorefname{property}{Property}
\makeautorefname{conjecture}{Conjecture}%
\makeautorefname{definition}{Definition}%
\makeautorefname{notation}{Notation}
\makeautorefname{notations}{Notations}
\makeautorefname{remark}{Remark}%
\makeautorefname{question}{Question}%
\makeautorefname{example}{Example}%
\makeautorefname{axiom}{Axiom}%
\makeautorefname{claim}{Claim}%
\makeautorefname{assumption}{Assumption}%
\makeautorefname{assumptions}{Assumptions}%
\makeautorefname{construction}{Construction}%
\makeautorefname{problem}{Problem}%
\makeautorefname{warning}{Warning}%
\makeautorefname{observation}{Observation}%
\makeautorefname{convention}{Convention}%
\makeautorefname{fact}{Fact}

\theoremstyle{plain}
\newtheorem{theorem}{Theorem}[section]
\newtheorem{corollary}{Corollary}[section]
\newtheorem{proposition}{Proposition}[section]
\newtheorem{lemma}{Lemma}[section]
\newtheorem{problem}{Problem}[section]

\newtheorem*{claim}{Claim}
\newtheorem{fact}{Fact}[section]

\theoremstyle{definition}
\newtheorem{definition}{Definition}[section]

\newtheorem{example}{Example}[section]
\newtheorem{notation}{Notation}[section]

\newtheorem{remark}{Remark}[section]

\newtheorem{observation}{Observation}[section]

\newtheorem{proviso}{Proviso}[section]

\makeatletter
\let\c@observation=\c@theorem
\let\c@corollary=\c@theorem
\let\c@proposition=\c@theorem
\let\c@lemma=\c@theorem
\let\c@problem=\c@theorem
\let\c@construction=\c@theorem
\let\c@conjecture=\c@theorem
\let\c@definition=\c@theorem
\let\c@notation=\c@theorem
\let\c@notations=\c@theorem
\let\c@example=\c@theorem
\let\c@axiom=\c@theorem
\let\c@property=\c@theorem
\let\c@assumption=\c@theorem
\let\c@warning=\c@theorem
\let\c@remark=\c@theorem
\let\c@scholium=\c@theorem
\let\c@nota=\c@theorem
\let\c@equation\c@theorem
\let\c@fact\c@theorem
\numberwithin{equation}{section}
\makeatother

\title{On approximate classification of theories}

\author{Alexander Burka}

\date{\today}

\address{Logic Group, University of California, Berkeley, 910 Evans Hall, Berkeley, CA 94720, United States}
\email{burka@berkeley.edu}


\begin{document}

\begin{abstract}
We propose a framework for model-theoretic classification theory in an approximate first-order setting and generalize some classical results.
\end{abstract}

\maketitle

\tableofcontents

\section{Introduction}

In its earliest forms \cite{morley}, \cite{bl}, model-theoretic stability was understood as a global property: a complete first-order theory $T$ is \emph{stable} if there exists a cardinal $\lambda \geq \aleph_0$ such that for every model $\M \models T$, $A \subset M$, and $n \in \omega$, $|A|\leq \lambda$ implies $|S_n(A,\M)|\leq \lambda$. A fundamental observation of Shelah's is that stability, even in its global formulation, is better examined locally, i.e. formula-by-formula: a partitioned formula $\phi = \phi(x,y)$ is \emph{stable} in $T$ if there exists $\lambda \geq \aleph_0$ such that for every model $\M \models T$ and $A \subset M$, if $|A|\leq \lambda$, then the space $S_\phi(A,\M)$ of complete $\phi$-types over $A$ has cardinality at most $\lambda$ \cite{ct}. Shelah proves that stability of a formula is equivalent to a number of properties: $\phi$ is stable in $T$ if and only if the $\phi$-$2$-rank $R(x=x,\phi,2)$ is finite, if and only if every complete $\phi$-type is definable over its own domain, if and only if $\phi$ does not have the order property, etc. This approach not only recovers the global formulation ($T$ is stable if and only if every partitioned formula is stable in $T$), but also lends itself to fruitful applications elsewhere, notably in extremal graph theory. Indeed, the dividing line between stable and unstable theories can be understood as a structure/randomness dichotomy in the class of complete first-order theories. Such themes were developed in parallel, albeit independently, in the extremal graph theory community, and it was not until recent years that direct connections were drawn. The celebrated stable regularity lemmas mark the first such contact: using quantitative, finitary refinements of local stability-theoretic notions, Malliaris--Shelah show that stable graphs, i.e. those whose edge relation forbids half-graphs of height $k$ for some $k$, admit $\epsilon$-regular decompositions into polynomially (in $\epsilon^{-1}$) many pieces without irregular pairs \cite{rlfsg}. As a consequence, the authors prove that the class of finite $k$-stable graphs has the Erd\H{o}s--Hajnal property: there exists $\delta = \delta_k > 0$ such that every finite $k$-stable graph $G$ contains a clique or anticlique of size $|G|^\delta$. Soon after, alternative proofs of stable regularity and stable Erd\H{o}s--Hajnal emerged that implemented \emph{bona fide} stability-theoretic arguments through the lens of pseudofinite model theory \cite{cs}, \cite{mp}. 

Stability is a sensitive property, yet the graph removal lemma suggests that applications of the method to problems in extremal combinatorics like the ones mentioned above ought to be agnostic to perturbations by negligible errors. (For an elaboration of this point, see \cite{mppw}.) To address this issue, various authors have proposed relaxations of the classical definition of stability that are robust to such perturbations but still lead to good structure theorems, at least at a generic scale. Perhaps the earliest approach to this sort of extension of classical stability theory is due to Hrushovski in his work on approximate subgroups \cite{approxsg}, where an invariant S1 ideal (e.g., the null ideal of an invariant Keisler measure) plays the role of Shelah's forking ideal. In \cite{cm}, Coregliano--Malliaris consider \emph{almost stable graphons} and establish strong connections to the Erd\H{o}s--Hajnal conjecture and approximations thereof. In the context of graded probability spaces, Chernikov--Towsner develop the notion of \emph{$\mu$-stability} towards strengthening regularity lemmas for stable graphs to hypergraphs \cite{chernikov–towsner}. Martin-Pizarro--Palacin--Wolf consider the related notion of \emph{robust stability} in the context of ultraproducts of finite groups \cite{mppw}, and the recent preprint \cite{giron} of Gir\'{o}n substantially extends this approach to arbitrary invariant S1 ideals in a general first-order setting. 

The present paper contributes to this growing body of literature from a complementary perspective. The common feature of the approaches to approximate stability in \cite{cm}, \cite{chernikov–towsner}, \cite{mppw}, and \cite{giron} is the negligibility of the set of half-graphs with respect to an invariant S1 ideal. In \cite{mppw}, for instance, a relation $R(x,y)$ definable in an ultraproduct $\frak{G}$ of finite structures is called \emph{robustly $k$-stable} if the set 
\[ \left\{ (a_0,\dots,a_{k-1},b_0,\dots,b_{k-1}) \in \frak{G}_x^k \times \frak{G}_y^k : \frak{G} \models \bigwedge_{i,j < n} R(a_i,b_j)^{\text{if } i < j} \right\} \]
of $k$-length half-graphs induced by $R$ has measure $0$ with respect to the appropriate power of the nonstandard Loeb measure on $\frak{G}$. The null ideals of such measures are S1; it is precisely this property that simplifies the concomitant forking calculus and leads to good generalizations of classical stationarity principles.

Not all ideals which occur in practice are S1: in the context of pseudofinite model theory, the ideal of definable subsets of coarse pseudofinite dimension less than $1$ is one such example \cite{approxsg}. Nonetheless, it is $\bigvee$-definable in a suitable language, and this setting represents the main case of interest in the present work.

We begin with the observation that the Stone space of a quotient Boolean algebra $B/I$ is homeomorphic to the closed subspace of $S(B)$ of \emph{$I$-wide ultrafilters}, i.e. those ultrafilters which avoid the ideal $I$. Intuitively, such ultrafilters are nonnegligible with respect to $I$, viewed as a notion of smallness in $B$. When $B = L_x(\frak{U})$ is the Boolean algebra of parametrically definable subsets (in a given context $x$) of a model $\frak{U}$ of a complete first-order theory $T$, this observation leads to the notions of a \emph{wide formula} and a \emph{wide type}.

If $I$ is $\bigvee$-definable and $\frak{U}$ is sufficiently saturated, we show that $I$ determines a functor 
\[ \M \mapsto \left(L_x(\M),I(\M)\right) \]
from the category of models of $T$ with elementary embeddings to the category of Boolean algebras with a distinguished ideal. The syntactic constructions that give rise to such functors are called \emph{schemas of $\bigvee$-definable ideals} for $T$; these objects lay the technical foundation for the rest of the paper by formalizing what it means to have a coherent assignment of an ideal $I(\M) \subset L_x(\M)$ to every model $\M$ of $T$.

Shelah's local $2$-rank naturally generalizes to this setting. In the classical setup, a definable set has rank at least $0$ if and only if it is nonempty; here, we localize to $I$ and say that a definable set $X$ has rank at least $0$ if and only if it is not in the ideal $I(\frak{M})$ for any model $\frak{M} \models T$ containing parameters over which $X$ is defined. When $I$ is $\bigvee$-definable, we obtain the following generalization of Shelah's stable formula theorem (cf. \cite{ct}, Theorem 2.2):

\begin{theorem}
Let $I$ be a schema of $\bigvee$-definable ideals for a complete theory $T$ and $\phi = \phi(x,y)$ a partitioned formula. The following are equivalent:
\begin{enumerate}[a.]
\item For every $\lambda \geq 2^{|L|}$, $\M \models T$, and $A \subset M$, if $|A| \leq \lambda$, then $\left|S^I_\phi(A,\M)\right| \leq \lambda$.
\item There exists $\lambda \geq \aleph_0$ such that for every $\M \models T$ and $A \subset M$, if $|A| \leq \lambda$, then $\left|S^I_\phi(A,\M)\right| \leq \lambda$.
\item $R^I(x=x,\phi,2) < \infty$.
\item $\Tree[I,\phi,\omega]$ is inconsistent with $T$.
\item For every $\omega$-saturated $\M \models T$ and $p \in S_\phi(\M)$, if $p$ is $I$-wide, then $p$ is definable.
\item No model $\M \models T$ contains sequences $\b{a} \in \M_x^\omega$, $\b{b} \in \M_y^{\omega}$ such that for all $i \in \omega$, $\tp(a_i/\b{b})$ is $I$-wide and for all $i$, $j\in \omega$ $\M \models \phi(a_i,b_j)$ holds iff $i < j$.
\end{enumerate}
\end{theorem}

We say that \emph{$\phi$ is $I$-stable in $T$} if any of these equivalent conditions holds. 

In \cite{approxeq}, Hrushovski considers the Fubini product $I \otimes J$ of invariant ideals, and shows that it behaves particularly well in the $\bigvee$-definable case. We extend these results and combine them with the preceding theorem to obtain a good notion of global stability modulo a schema of $\bigvee$-definable ideals:

\begin{theorem}
The following are equivalent:
\begin{enumerate}[a.]
\item Every partitioned formula $\phi = \phi(x,y)$ is $I^{\otimes l(x)}$-stable.
\item Every partitioned formula $\phi = \phi(x,y)$ with $l(x) = 1$ is $I$-stable.
\item For every $\lambda = \lambda^{|T|}$, $\M \models T$, and $A \subset M$, if $|A| \leq \lambda$ then $\left|S^I_1(A,\M)\right| \leq \lambda$.
\item For every $\lambda = \lambda^{|T|}$, $\M \models T$, $A \subset M$, and $n \in \omega \setminus\left\{0\right\}$, if $|A| \leq \lambda$, then $\left|S^I_n(A,\M)\right|\leq\lambda$.
\end{enumerate}
\end{theorem}

The final section of the paper develops a suitable theory of independence for wide types. Classically, a formula \emph{divides} if it is inconsistent along an indiscernible sequence in the type of its parameters. In our setting, we say that a formula \emph{$I$-divides} if it is $I$-thin (i.e., not $I$-wide) along an indiscernible sequence in the type of its parameters. Control over the lengths of uniformly $I$-dividing sequences gives rise to a good notion of approximate simplicity. In particular, we prove the following:

\begin{theorem}
The following are equivalent:
\begin{enumerate}[a.]
\item $T$ is $I$-simple: for every $\phi(x,y)$, $k \in \omega$, and $A$ there are no infinite uniformly $\phi$-$k$-$I$-dividing sequences over $A$.
\item $T$ satisfies wide local character with respect to $I$: for every $A \subset \frak{U}$, $p(x) \in S^I(A)$,\footnote{$x$ may be an arbitrary (finite) tuple of free variables.} $p(x)$ does not $I$-divide over some $A_0 \subset A$ with $|A_0|\leq|T|$.
\item For every $\omega$-saturated $\M \models T$ and $p(x) \in S^I(\M)$, $p(x)$ does not $I$-divide over some $A_0 \subset M$ with $|A_0|\leq|T|$.
\end{enumerate}
\end{theorem}

As a corollary, we deduce that $I$-stability implies $I$-simplicity. We then prove some additional properties of $I$-forking in $I$-simple theories, such as the following generalization of Kim's lemma \cite{kim}:

\begin{theorem}
Let $T$ be $I$-simple. Let $\pi(x,y)$ be a partial type over $A$ and $\b{b} = (b_i : i \in \omega)$ an infinite $I$-Morley sequence over $A$. If $\bigcup_{i \in \omega} \pi(x,b_i)$ is $I$-wide, then $\pi(x,b_0)$ does not $I$-divide over $A$. 
\end{theorem}

It follows that $I$-forking and $I$-dividing coincide in $I$-simple theories. We conclude with an open problem regarding the symmetry of $I$-independence in $I$-simple theories.

\medskip

\noindent\textbf{Acknowledgements.} The results of this paper are part of the author's Ph.D. dissertation, supervised by Tom Scanlon. The author thanks him for his guidance throughout the course of this project. The author also thanks Maryanthe Malliaris and Julia Wolf for helpful conversations. Finally, the author thanks L.~C. Brown for contributing Example \ref{type-WUFT counterexample}.

This research has been partially supported by the National Science Foundation (Grant No.\ DMS-22010405).
\subsection{Preliminaries}

The first two parts of Section 2 deal with abstract Boolean algebras and their quotients. In this connection it is useful to define the symmetric difference $a \Delta b := \left( a \wedge b' \right) \vee \left( a' \wedge b \right)$. We recall the celebrated

\begin{theorem}
(Stone representation theorem; \cite{stone}) Every Boolean algebra $B$ is isomorphic to the algebra $\text{Clop}\left(S\left(B\right)\right)$ of clopen subsets of its Stone space $S(B)$. In fact, the contravariant functors $B \mapsto S(B)$, $X \mapsto \text{Clop}(X)$ determine an anti-equivalence of categories between the category of Boolean algebras with lattice homomorphisms and the category of Stone (i.e., compact, Hausdorff, and totally disconnected) topological spaces with continuous maps.
\end{theorem}

\noindent For more details, we refer the reader to a standard reference (e.g. \cite{halmos}).

The remaining portion of the paper deals with model-theoretic classification theory. Our notation is standard, mostly following \cite{ct} barring a few exceptions which we clarify below.

Fix a (possibly multisorted) first-order language $\L$, an $\L$-structure $\M$, and a subset $A \subset M$. For $\phi \in \L$, we use the convention $\phi^1 := \phi$ and $\phi^0 := \neg\phi$. $\L_x$ denotes the set of partitioned $\L$-formulas\footnote{Formally, a partitioned formula is a tuple $(\phi,x,y)$ such that $\phi$ is an $\L$-formula, $x$ and $y$ are disjoint tuples of variables in $\L$, and the free variables of $\phi$ are among those in the tuple $x^\frown y$. If $(\phi,x,y)$ is a partitioned formula, we write $\phi = \phi(x,y)$ and say that $\phi$ is a partitioned $\L$-formula in the \emph{context} $x$ and \emph{parameters} $y$.} in the context $x$; $\L_x(A)$ denotes the set of formulas of the expanded language $\L(A)$ whose free variables are among those in $x$. Note that under this convention, $\L_x$ and $\L_x(\emptyset)$ are distinct. We write $\M_x$ for the set of $x$-tuples of elements of the domain of $\M$.

By a definable subset of a model, we mean definable with parameters. The quotient $\L_x(A)/{\sim_\M}$, where $\phi\sim_\M\psi$ if and only if $\M\models\forall x(\phi\leftrightarrow\psi)$, acquires the structure of a Boolean algebra $L_x(A,\M)$ whose Boolean operations are defined in terms of the corresponding logical operations on representatives; it is naturally isomorphic to the subalgebra $\{ \phi(\M,a) : \phi(x,y) \in \L_x(A) \}$ of $\mathcal{P}(\M_x)$. For a collection $\Delta \subset \L_x$, $L_\Delta(A,\M)\subset L_x(A,\M)$ denotes the subalgebra generated by $\left\{[\phi(x,a)]_{\sim\M}:\phi(x,y)\in\Delta,a\in A^{|y|}\right\}$.

Let $S_x(A,\M)$ be the set of complete $x$-types over $A$ with respect to $\M$. As it only depends on the $\L(A)$-theory of $\M$, we often omit $\M$ when it no confusion can arise. The map $p(x) \mapsto \left\{[\phi(x)_{\sim_\M}] : \phi(x) \in p(x) \right\}$ identifies $S_x(A,\M)$ with the Stone space $S(L_x(A,\M))$. $\Delta$-types are defined analogously using $\L_\Delta(A)$; the space of complete $\Delta$-types is $S_\Delta(A,\M)$. 

For an infinite $\kappa$, $\M$ is $\kappa$-saturated if every type in finitely (equivalently, $\leq \kappa$) many variables over $A\subset M$ with $|A|<\kappa$ is realized in $\M$; strongly $\kappa$-homogeneous if any tuples $a,b\in\M_x$ with $\tp(a/A)=\tp(b/A)$ are related by some $f\in\Aut(\M/A)$;and $\kappa$-universal if every elementarily equivalent model of cardinality $\leq \kappa$ admits an elementary embedding into $\M$. A model is $\kappa$-saturated and strongly homogeneous if both hold. When $\kappa = \aleph_0$, we use $\omega$ in its stead. The following fact is standard:
\begin{fact}[\cite{ct,hodges}]
\begin{enumerate}[i.]
\item $\kappa$-saturated models are $\kappa$-universal.
\item Every theory has $\kappa$-saturated, strongly homogeneous models for all $\kappa$.
\item Every structure has an $\omega$-saturated elementary extension of size $\le |M|+2^{|\L|}$.
\end{enumerate}
\end{fact}
A definable set $D\subset\M$ is algebraic if $|D|<\aleph_0$; $a$ is algebraic over $A$ if $\tp(a/A)$ contains an algebraic formula. Write $\acl(A)$ for the set of all such elements. If $\M$ is strongly $\kappa$-homogeneous and $|A|<\kappa$, then $D$ is $A$-definable if and only if $f(D)=D$ for all $f\in\Aut(\M/A)$, and $a\in\acl(A)$ if and only if its orbit under $\Aut(\M/A)$ is finite.

The \emph{$n$-Ehrenfeucht--Mostowski type} of the sequence $\b{a}$ over $A$ is given by 
\[ \EM_n(\b{a}/A) := \left\{ \phi(x_1,\dots,x_n) \in \L(A) : \frak{M} \models \phi(a_{j_1},\dots,a_{j_n}) \text{ for all } j_1 < \cdots < j_n \in J \right\} \]
and the \emph{Ehrenfeucht--Mostowski type} of $\b{a}$ over $A$ is defined as
\[\EM(\b{a}/A) := \bigcup_{n \in \omega} \EM_n(\b{a}/A).\]

\begin{fact}
(Standard Lemma; \cite{tz}) In every $|K| + |A|^+$-saturated elementary extension $\N \succ \M$ there exists an $A$-indiscernible sequence $\b{a}' = (a_i' : i \in K) \in \left(\M_x\right)^K$ which realizes $\EM(\b{a}/A)$.
\end{fact}

For more sophisticated coloring arguments, we will require a stronger version of Ramsey's theorem:

\begin{fact}
(Erd\H{o}s--Rado--Shelah; \cite{tz}) Every $\mu$-coloring of the set of $n+1$-element subsets of $\beth^+_n(\mu)$ admits a homogeneous set of cardinality $\mu^+$.

Let $\frak{M}$ be a sufficiently saturated structure. Then for every small $A \subset M$, there exists $\lambda$ such that for every infinite linear order $J$ of cardinality $\lambda$ and sequence $\b{a} \in \left(\M_x\right)^J$ there exists an $A$-indiscernible sequence $\b{b} \in \left(\M_x\right)^\omega$ such that for all $j_1 < \cdots < j_n < \omega$ there exist $i_1 < \cdots < i_n \in J$ such that $\tp(a_{i_1},\dots,a_{i_n}/A) = \tp(b_{j_1},\dots,b_{j_n}/A)$. 
\end{fact}

We recall another combinatorial

\begin{fact}
(Erd\H{o}s--Makkai; \cite{tz}) Let $X$ be an infinite set, and let $\mathcal{S} \subset \P(X)$. If $|\mathcal{S}| > |X|$, then there exist $\b{a} \in X^\omega$ and $(S_i : i \in \omega) \in \mathcal{S}^\omega$ such that for all $i, j \in \omega$ either 
\begin{enumerate}[i.]
\item $a_i \in S_j \Leftrightarrow i < j$, or
\item $a_i \in S_j \Leftrightarrow j < i$.
\end{enumerate}
\end{fact}

Finally, we recall the notion of a definable type from \cite{ct}. Let $\phi = \phi(x,y) \in \L_x$, let $\pi(x)$ be a partial type with respect to $\M$, and let $A \subset M$. $\pi(x)$ is \emph{$\phi$-$A$-definable} if there exists $\psi(y) \in \L_y(A)$ such that for all $c \in \M_y$, if $\phi(x,c) \in \pi(x)$ then $\M \models \psi(c)$ and if $\neg\phi(x,c) \in \pi(x)$ then $\M \models \neg\psi(c)$. $\pi(x)$ is \emph{$A$-definable} if it is $\phi$-$A$-definable for every $\phi \in \L_x$. $\pi(x)$ is \emph{definable} if it is $\dom \pi(x)$-definable.

\section{Wide ultrafilters}
\subsection{Quotient algebras and their ultrafilters}

If $B$ is a Boolean algebra and $I \subset B$ is an ideal, then $I$ uniquely determines a congruence $\equiv$ on $B$ given by $a \equiv b$ if and only if $a \Delta b \in I$. As such the set of congruence classes $[a]_I \in B/I$ inherits the structure of a Boolean algebra, and the map $\pi : B \to B/I$ is an epimorphism in the category of Boolean algebas. By Stone's representation theorem, $\pi$ induces a monomorphism $\pi^* : S(B/I) \to S(B)$ via $\pi^*\left(\U\right) = \pi^{-1}\U$. But continuous injections from compact to Hausdorff spaces are embeddings, so $\pi^*$ is in fact a homeomorphism from $S(B/I)$ onto its image. This subspace of $S(B)$ admits a concrete description:

\begin{definition}
An ultrafilter $\U \in S(B)$ is \emph{wide} with respect to $I$ if $\U \cap I = \emptyset$. Let $S^I(B)$ denote the closed subspace of $I$-wide ultrafilters on $B$.
\end{definition}

A subset $A \subset B$ is \emph{$I$-saturated} if it is saturated with respect to the associated congruence, i.e. $a \in A$ and $a \Delta b \in I$ implies $b \in A$. Clearly every $\U \in S^I(B)$ is $I$-saturated: if $a \in \U$ and $b \equiv a$ but $b \notin \U$, then maximality implies $a \wedge b' \in \U$ and $b \equiv a$ implies $a \wedge b' \equiv \bot$, i.e. $a \wedge b' \in I$. Using this observation, it is straightforward to verify that $\pi^*$ maps $S(B/I)$ onto $S^I(B)$.

Recall that the Boolean prime ideal theorem asserts the existence of ultrafilters extending subsets of $B$ with the finite meet property (FMP). It also shows that the same holds for wide ultrafilters with respect to subsets $X \subset B$ with the \emph{$I$-finite meet property} ($I$-FMP), i.e. those $X$ for which $\bigwedge X_0 \notin I$ for all finite $X_0 \subset X$.

\begin{proposition}
\label{FMP_correspondence}
The set of $I$-FMP, $I$-saturated subsets of $B$ bijects with the set of FMP subsets of $B/I$ via the maps $X \mapsto \pi X$ and $Y \mapsto \pi^{-1} Y$.
\end{proposition}
\begin{proof}
If $X \subset B$ is $I$-FMP, then $Y := \pi X$ is FMP. If $Y \subset B/I$ is FMP, then $X := \pi^{-1}Y$ is $I$-saturated and $I$-FMP. So these maps are well-defined. Next, we show the correspondence is bijective when restricted to the $I$-FMP subsets of $B$ which are also $I$-saturated. For $Y \subset B/I$, surjectivity of $\pi$ implies $\pi \pi^{-1}Y = Y$. On the other hand, suppose $X \subset B$ is $I$-saturated and $I$-FMP. $X \subset \pi^{-1}\pi X$ is automatic, and the other inclusion uses the definition of $I$-saturation.
\end{proof}

\begin{proposition}
\label{saturation_is_FMP}
For $X \subset B$, let $\hat{X} := \bigcup_{x\in X}[x]_I$ be its $I$-saturation. If $X$ is $I$-FMP, then so is $\hat{X}$.
\end{proposition}
\begin{proof}
Let $\hat{X}_0 \subset \hat{X}$ be finite. For each $x \in \hat{X}_0$, choose some $x^*\in X$ such that $x \equiv x^*$ and put $X_0 = \{x^* : x \in \hat{X}_0 \}$. If $X$ has the $I$-finite meet property, then we have $\bigwedge X_0 \notin I$. But $\bigwedge \hat{X}_0 \equiv \bigwedge X_0$; since $I$ is itself $I$-saturated, we conclude $\bigwedge \hat{X}_0 \notin I$, as desired.
\end{proof}

\begin{corollary}
\label{I-BPI}
Any $I$-FMP subset of $B$ extends to an $I$-wide ultrafilter on $B$.
\end{corollary}
\begin{proof}
Let $X \subset B$ have the $I$-finite meet property. By Proposition \ref{saturation_is_FMP}, its $I$-saturation $\hat{X}$ must also have the $I$-finite meet property. In view of Proposition \ref{FMP_correspondence}, $\hat{X}$ uniquely corresponds to an FMP subset $\hat{Y} := \pi \hat{X}$ of $B/I$. The upward closure
\[\left\{a \in B/I : \bigwedge Y_0 \leq a \text{ for some finite } Y_0 \subset \hat{Y}\right\}\]
of the $\pi$-system generated by $\hat{Y}$ is a filter on $B/I$ containing $\hat{Y}$, which thus extends to an ultrafilter $\V$ on $B/I$ by Zorn's lemma. But $\V$ in turn corresponds to a wide ultrafilter $\U := \pi^{-1}\V$ on $B$ as noted in the observation above. Moreover, we have
\[X \subset \hat{X} = \pi^{-1}\pi \hat{X} \subset \pi^{-1}\V = \U.\]
This completes the proof.
\end{proof}
\subsection{Wide types}

Fix an $\L$-structure $\M$, a context $x$, and an ideal $I \subset L_x(\M)$. For $\Delta \subset \L_x$ and $A \subset M$, define $I\restrict_{\Delta,A} := I \cap L_\Delta(A,\M)$; we omit $\Delta$ if $\Delta = \L_x$ and $A$ if $A = M$.

\begin{definition}
\begin{enumerate}[i.]
\item A formula $\phi(x) \in \L_x(M)$ is \emph{wide} with respect to $I$ if $\left[\phi(x)\right]_{\sim_\M} \notin I$. Otherwise, $\phi(x)$ is \emph{thin}. $\phi(x)$ is \emph{co-wide} if $\neg\phi(x)$ is wide, and is \emph{full} if $\neg\phi(x)$ is thin.
\item A partial type $p(x)$ is \emph{wide} if $\bigwedge p_0(x)$ is wide for every finite $p_0(x) \subset p(x)$. Otherwise, $p(x)$ is \emph{thin}.
\item $S^I_\Delta(A,\frak{M})$ denotes the set of $I$-wide complete $\Delta$-types over $A$ (with respect to $\frak{M}$) and $S^I_x(A,\frak{M})$ is the set of $I$-wide complete $x$-types over $A$ (with respect to $\frak{M})$. If $\frak{M}$ is clear from context, we omit it and write $S^I_\Delta(A) := S^I_\Delta(A,\frak{M})$ and $S^I_x(A) := S^I_x(A,\frak{M})$ respectively.
\item For $a \in \M_x$, $A \subset M$, $a$ is \emph{wide over A} if $\tp(a/A)$ is wide.
\end{enumerate}
\end{definition}

Note that under the usual identification $S_x(A,\M) \to S\left(L_x(A,\M)\right)$, $S^I_x(A, \frak{M})$ corresponds to the space $S^{I\restrict_{\Delta,A}}(L_x(A,\frak{M}))$ of $I$-wide ultrafilters on $L_x(A,\frak{M})$, and likewise for $S^I_\Delta(A,\frak{M})$.

The following are immediate consequences of Corollary \ref{I-BPI}.

\begin{corollary}
    \label{wide types}
\begin{enumerate}[i.]
\item Let $\pi(x)$ be a partial $\Delta$-type over $A$. If $\pi(x)$ is $I$-wide, then it extends to an $I$-wide complete $\Delta$-type over $A$. In particular, every partial wide type over $A$ extends to a complete wide type over $A$.
\item If $\M$ is $\aleph_0+|A|^+$-saturated and $\pi(x)$ is a wide partial type over $A$, then there exists $a \in \M_x$ which is wide over $A$ and such that $\M \models \pi(a)$.
\end{enumerate}
\end{corollary}
\subsection{$\bigvee$-definable ideals}

Let $\M \prec \frak{N}$ be an elementary extension of $\L$-structures, $A \subset M$ a set of parameters, and $I \subset L_x(\M)$ an ideal. Suppose that $\M$ is $\aleph_0+|A|^+$-saturated.

\begin{definition}
$I$ is \emph{$A$-$\bigvee$-definable} if for every $\L$-formula $\psi(x,y)$, the set $\left\{c \in \M_z : \psi(x,c) \text{ is }I\text{-wide} \right\}$ is type-definable over $A$. Let $\neg I x.\psi(x,y)$ be any set of $\L_y(A)$-formulas which defines this subset of $\M$. 
\end{definition}

\begin{observation}\label{V observation}
Suppose $I$ is $A$-$\bigvee$-definable. If $\M \models \forall x\left( \psi(x,c) \leftrightarrow \theta(x,d)\right)$, we have
\begin{align*}
\M \models \neg I x.\psi(x,c) & \Longleftrightarrow \psi(x,c) \text{ is }I\text{-wide} \\
                              & \Longleftrightarrow \theta(x,d) \text{ is }I\text{-wide} \\
                              & \Longleftrightarrow \M \models \neg I x.\theta(x,d).
\end{align*}
Because $\M$ is $\aleph_0 + |A|^+$-saturated, the same holds for $\frak{N}$. Indeed, assume for a contradiction that $\psi(x,c)$, $\theta(x,d) \in \L_x(N)$ are such that $\frak{N} \models \neg I x.\psi(x,c)$ and $\frak{N} \not\models \neg I x.\theta(x,d)$, yet $\frak{N} \models \forall x\left(\psi(x,c) \leftrightarrow \theta(x,d)\right)$. Since $\frak{N} \not\models \neg I x.\theta(x,d)$, there is a finite $q(z) \subset \neg I x.\theta(x,z)$ such that $\frak{N} \models \neg\bigwedge q(d)$. Consider the collection of $\L(A)$-formulas
\[\Gamma(y^\frown z) := \neg I x.\psi(x,y) \cup \left\{\neg\bigwedge q(z), \forall x\left(\psi(x,y)\leftrightarrow \theta(x,z)\right)\right\}.\]
$\Gamma(y^\frown z)$ is a partial type over $A$ because it is satisfied (in $\frak{N}$) by $(c,d)$. The assumption that $\M$ is an $\aleph_0 + |A|^+$-saturated elementary substructure of $\frak{N}$ delivers $c',d' \in M$ such that $\M \models \forall x\left(\psi(x,c')\leftrightarrow\theta(x,d')\right)$, $\M \models \neg I x.\psi(x,c')$, and $\M \not\models \neg I x.\theta(x,d')$. So $\psi(x,c')$ is $I$-thin and $\theta(x,d')$ is $I$-wide, yet they define the same subset of $\M$---contradiction.

It follows that the \emph{a priori} distinct subsets of $L_x(\frak{N})$ given by
\begin{align*}
I_0(\frak{N}) & := \left\{ X \in L_x(\frak{N}) : \frak{N}\not \models \neg I x.\theta(x,d) \text{ for \emph{every} } \theta(x,d) \in \L_x(N) \text{ whose } \sim_\frak{N}-\text{class is } X\right\} \text{ and } \\
I_1(\frak{N}) & := \left\{ X \in L_x(\frak{N}) : \frak{N}\not \models \neg I x.\theta(x,d) \text{ for \emph{some} } \theta(x,d) \in \L_x(N) \text{ whose } \sim_\frak{N}-\text{class is } X\right\}
\end{align*}
are in fact the same; let $I(\frak{N})$ denote the common value.
\end{observation}

\begin{lemma}\label{V extensions}
Let $\M \prec \frak{N}$ be an elementary extension, $A \subset M$, and $I \subset L_x(\M)$ an ideal. Suppose $\M$ is $\aleph_0 + |A|^+$-saturated. If $I$ is $A$-$\bigvee$-definable, then $I(\frak{N})$ is an $A$-$\bigvee$-definable ideal of $L_x(\frak{N})$ and $I = i_*^{-1}I(\frak{N})$.
\end{lemma}
\begin{proof}
To ease notation, we assume $A = \emptyset$.

Let $D,E \in L_x(\frak{N})$. By the preceding observation, we may choose arbitrary representatives: let $\psi = \psi(x,y)$, $\theta = \theta(x,z) \in \L_x$ and assume $D$, $E$ are represented by $\psi(x,c)$ and $\theta(x,d)$, respectively. Suppose that $D \leq E$,\footnote{I.e., $\psi\left(\frak{N},c\right) \subset \theta\left(\frak{N},d\right)$.} so that $\frak{N} \models \forall x \left( \psi(x,c) \to \theta(x,d)\right)$, and assume $\frak{N} \models \neg I x.\psi(x,c)$. We want to show $\frak{N} \models \neg I x.\theta(x,d).$ For a contradiction, assume not. Then there is a finite $q(z) \subset \neg I x.\theta(x,z)$ such that $\frak{N} \models \neg \bigwedge q(d)$. Thus, the following set of formulas is satisfiable in $\frak{N}$:
\[\left\{ \forall x\left(\psi(x,y)\to\theta(x,z)\right) \right\} \cup \neg I x.\psi(x,y) \cup \left\{\neg\bigwedge q(z) \right\}.\]
The assumption that $\M$ is $\aleph_0 + |A|^+$-saturated delivers $c' \in \M_y$, $d' \in \M_z$ such that $\M \models \forall x\left( \psi(x,c')\to\theta(x,d')\right)$ and $\M \models \neg I x.\psi(x,c')$ yet $\M \models \neg \bigwedge q(d')$. As $c'$ satisfies $\neg I x.\psi(x,y)$ in $\M$, we have $[\psi(x,c')]_{\sim_{\M}} \notin I$ by definition of $\bigvee$-definability. But $[\psi(x,c')]_{\sim_{\M}}\leq [\theta(x,d')]_{\sim_{\M}}$ holds, and this implies $[\theta(x,d')]_{\sim_{\M}} \notin I$ because ideals are closed downwards. Therefore $\M \models \neg I x.\theta(x,d')$, yet $\M \models \neg \bigwedge q(d')$ for some finite $q(z) \subset \neg I x.\theta(x,z)$, which is absurd.

A similar compactness argument shows that $\frak{N} \not\models \neg I x.\psi(x,c)$ and $\frak{N}\not\models \neg I x.\theta(x,d)$ implies $\frak{N}\not\models \neg I x.\left( \psi(x,c)\vee\theta(x,d) \right)$,
i.e. $I(\frak{N})$ is closed under finite disjunctions. Likewise $[x\neq x]_{\sim_\N} \in I(\N)$.

Finally, we have
\[[\psi(x,c)]_{\sim_\M} \notin I \Longleftrightarrow \M \models \neg Ix.\psi(x,c) \Longleftrightarrow \N \models \neg Ix.\psi(x,c) \Longleftrightarrow [\psi(x,c)]_{\sim_\N} \notin I(\N)\]
and so $I = i_*^{-1}I(\frak{N})$.
\end{proof}

\begin{definition}
A \emph{preschema of partial types on $\L_x$} is a function
\[I : \L_x \to \P(\L), \psi(x,y) \mapsto \neg I x. \psi(x,y)\]
such that for each $\psi \in \L_x$ and $\theta \in \neg I x.\psi$, we have $\fv(\theta) \subset \fv(\psi) \setminus x$.

Let $I$ be a preschema of partial types on $\L_x$ and let $\M$ be an $\L$-structure. $I$ is \emph{$\sim_{\M}$-invariant} if for all $\psi(x,c),\theta(x,d) \in \L(M)$, if $\M \models \forall x \left( \psi(x,c) \leftrightarrow \theta(x,d)\right)$, then 
\[ \M \models \neg I x.\psi(x,c) \Longleftrightarrow \M \models \neg I x.\theta(x,d).\]
\end{definition}

Note that if $I$ is $\sim_\M$-invariant, then the set of $D \in L_x(\M)$ for which $\M \not\models \neg I x.\psi(x,c)$ for some $\psi(x,c) \in \L_x(M)$ with $[\psi(x,c)]_{\sim_\M} = D$ coincides with the set of $D \in \L_x(\M)$ for which $\M \not\models \neg Ix.\psi(x,c)$ for every such $\psi(x,c)$. Accordingly, $I(\M)$ denotes the common value.

\begin{definition}
$I$ \emph{$\bigvee$-defines an ideal of $L_x(\M)$} if it is $\sim_\M$-invariant and $I(\M) \subset L_x(\M)$ is an ideal.

Let $T$ be a complete $\L$-theory. A \emph{schema of $\bigvee$-definable ideals for $T$} in the context $x$ is a preschema $I : \L_x \to \P(\L)$ of partial types which $\bigvee$-defines an ideal of $L_x(\M)$ for every model $\M \models T$. We refer to the ideal $I(\M)$ as the \emph{$\M$-points of $I$}.

Let $I : \L_x \to \P(\L)$ be a schema of $\bigvee$-definable ideals for $T$. For a set $A$ of parameters in a model $\M$ of $T$, define $S^{I}_x(A,\M) := S^{I(\M)}_x(A,\M)$. Similarly, if $\Delta$ is a set of partitioned formulas in the context $x$, define $S^I_\Delta(A,\M) := S^{I(\M)}_\Delta(A,\M)$. 
\end{definition}

\begin{observation}
\begin{enumerate}[i.]
\item If $I$ $\bigvee$-defines an ideal of $L_x(\M)$, then the ideal $I(\M)$ is $\emptyset$-$\bigvee$-definable.

\item If $I$ is a schema of $\bigvee$-definable ideals for $T$, $\M$, $\frak{N}$ are models, and $f : \M \to \frak{N}$ is elementary, then $I(\M) = f_*^{-1}I(\frak{N})$.

\item Let $I$ be a schema of $\bigvee$-definable ideals for $T = \Th(\M) = \Th(\N)$. Let $p(x)$ be a partial type over a common set of parameters $A \subset M \cap N$. If $\M \equiv_A \N$, then $p(x)$ is $I(\M)$-wide if, and only if, it is $I(\frak{N})$-wide. In this case we simply say that $p(x)$ is \emph{$I$-wide}. Following the same reasoning, we find $S^{I}_x(A,\M) = S^{I}_x(A,\frak{N})$. The common value is thus denoted by $S^I_x(A)$ if the ambient model $\M$ is understood. Likewise for $S^I_\Delta(A)$ for $\Delta \subset \L_x$.

\item Suppose $\M$ is an $\omega$-saturated model of $T$ and $I \subset L_x(\M)$ is an $\emptyset$-$\bigvee$-definable ideal. For each $\psi = \psi(x,y)$, let $\pi_\psi(y)$ be any partial type over $\emptyset$ which defines $\left\{ a \in \M_y : [\psi(x,c)]_{\sim_\M} \notin I \right\}$. Then the mapping $\L_x \to \P(\L), \psi \mapsto \pi_\psi$ is a preschema of partial types on $\L_x$ which $\bigvee$-defines the ideal $I$ of $L_x(\M)$. The next result shows that any such mapping is, in fact, a schema of $\bigvee$-definable ideals for $T$.
\end{enumerate}
\end{observation}

\begin{theorem} \label{V schemas}
Let $I : \L_x \to \P(\L)$ be a preschema of partial types, $T$ a complete $\L$-theory, and $\M \models T$ an $\omega$-saturated model of $T$. If $I$ $\bigvee$-defines an ideal of $L_x(\M)$, then $I$ is a schema of $\bigvee$-definable ideals for $T$.
\end{theorem}

\begin{proof}
Let $\frak{N} \models T$. By elementary amalgamation, there exists $\frak{U}\models T$ which admits elementary embeddings $f:\M\to\frak{U}$ and $g:\frak{N} \to \frak{U}$. (The proof of) Lemma \ref{V extensions} applies: since $I$ $\bigvee$-defines an ideal of $L_x(\M)$ and $\M$ is $\omega$-saturated, it $\bigvee$-defines an ideal of $L_x(\frak{U})$. 

It remains to check that $I$ $\bigvee$-defines an ideal of $L_x(\frak{N})$. First, we have to check that $I$ is $\sim_\frak{N}$-invariant. Let $\psi(x,c),\theta(x,d) \in \L_x(N)$ be such that $\frak{N} \models \forall x \left( \psi(x,c) \leftrightarrow \theta(x,d)\right)$, and assume that $\frak{N} \models \neg I x.\psi(x,c)$. Then, as $g : \frak{N} \to \frak{U}$ is elementary, we have $\frak{U} \models \forall x \left( \psi(x,g(a)) \leftrightarrow \psi(x,g(b)) \right)$ and $\frak{U} \models \neg I x.\psi(x,g(a))$. Since $I$ is $\frak{U}$-invariant, it follows, $\frak{U} \models \neg I x.\psi(x,g(b))$, whence $\frak{N} \models \neg I x.\theta(x,d)$. A similar argument demonstrates that because $I(\frak{U}) \subset L_x(\frak{U})$ is an ideal, $I(\frak{N})\subset L_x(\frak{N})$ is also an ideal, as desired.
\end{proof}

Let $I$, $I'$ be schemas of $\bigvee$-definable ideals for $T$ in the context $x$. $I$, $I'$ are \emph{equivalent modulo $T$} if $I(\M) = I'(\M)$ for all $\M \models T$. If $\frak{U}$ is a fixed $\omega$-saturated model of $T$, Theorem \ref{V schemas} says that equivalence classes of schemas of $\bigvee$-definable ideals for $T$ are in one-to-one correspondence with $\emptyset$-$\bigvee$-definable ideals of $L_x(\frak{U})$.

\begin{example} \label{basic V example}
Let $D_n(x) \in \L_x(\emptyset)$ be a countable family of $\emptyset$-definable predicates in the same context $x$. For $\psi(x,y) \in \L_x$ define
\[\neg I x. \psi(x,y) := \left\{ \exists x \left( \psi(x,y) \wedge \bigwedge_{n \in K} D_n(x) \right) : K \subset \omega \text{ finite} \right\}.\]
It is straightforward to show that the preschema $I : \L_x \to \P(\L)$ of partial types so defined is a schema of $\bigvee$-definable ideals for any $T$.
\end{example}

\begin{example}
Consider the preschema of partial types $I : \L_x \to \P(\L)$ given by 
\[\psi(x,y) \mapsto \neg Ix.\psi(x,y) := \left\{ \exists x_0\dots x_{n-1} \left( \bigwedge_{i < j < n} x_i \neq x_j \wedge \psi(x_i,y) \right) : n \in \omega \right\}.\]
Let $T$ be any complete $\L$-theory, $\M \models T$. Clearly $I$ is $\sim_\M$-invariant; we want to show $I(\M)$ is an ideal. If $\psi(x,c)$, $\theta(x,d) \in \L_x(\M)$ are such that $\M \not\models \neg I x.\psi(x,c)$ and $\M \not\models \neg I x.\theta(x,d)$, then $\psi(x,c)$ and $\theta(x,d)$ define finite subsets of $\M$; thus, $\psi(x,c) \vee \theta(x,d)$ also defines a finite subset of $\M$, so its $\sim_\M$-class is in $I(\M)$. Similarly, $I(\M)$ is closed downwards under implication and contains $\bot$. Thus $I$ is a $\bigvee$-definable ideal for $T$. 

$I(\M)$ consists of the definable subsets of $\M_x$ which are algebraic, so for $a \in \M_x$, we have $a \notin \acl(A)$ if and only if $\tp(a/A)$ is $I$-wide. 
\end{example}

\begin{example} \label{delta example}
Let $\L$ be a language of finite signature. Consider the expansion $\L_0^+$ of $\L$ by a new sort $\OF$ which carries the signature $\{ +, \cdot, -, 0, 1, < \}$ of ordered fields, a partial unary map $\log : \OF_{> 0} \to \OF$, and, for each context $x$ of the home sort $\bbH$ and $\psi = \psi(x,y) \in \L_x$, an $l(y)$-ary map $f_\psi : \bbH^{l(y)} \to \OF$.

We can repeat this process to produce an expansion $\L_1^+$ which has such function symbols $f_\psi$ for partitioned formulas $\psi = \psi(x,y) \in \left(\L_0^+\right)_x$ whose context is of sort $\bbH$, but whose parameters are of either sort. Iterating countably many times obtains the expansion $\L^+$; it has the property that for each $x : \bbH$ and $\psi = \psi(x,y) \in \left(\L^+\right)_x$, we have a map $f_\psi$ from the appropriate sorts in the domain to $\OF$. Moreover, $\L^+$ is still countable. For $\psi = \psi(x,y) \in \left(\L^+\right)_x$, we write $|\psi(x,y)| := f_\psi(y)$.

Let $(\frak{G}_n : n \in \omega)$ be an increasing (i.e., $|\frak{G}_n| < |\frak{G}_{n+1}|$) sequence of finite $\L$-structures. Each $\frak{G}_n$ has a canonical expansion to an $\L_0^+$-structure $(\frak{G}_n, \bbR)_0$ by interpreting $\bbH$ as $\frak{G}_n$, $\OF$ as the ordered real field with the logarithm, and for each $\psi = \psi(x,y) \in \L_x$ the operation $f_\psi$ as the map
\[\left(\frak{G}_n\right)_y \to \bbR,\text{ } c \mapsto \left|\psi(\frak{G}_n,c)\right|.\]
In turn, we obtain an expansion of each $\left(\frak{G}_n,\bbR\right)_0$ to an $\L^+$-structure $(\frak{G},\bbR)$ defined in the analogous way.

Take a nonprincipal ultrafilter $\U$ on $\omega$. The ultraproduct $\left(\frak{G},\bbR^\U\right) := \prod_{n \to \U} (\frak{G}_n, \bbR)$ consists of a continuum-sized $\L$-structure $\frak{G} = \prod_{n \to \U} \frak{G}_n$ in the homesort and a nonstandard continuum-sized elementary extension $\bbR^\U$ of the ordered real field with the logarithm, along with functions $f_\psi(y)$ which track nonstandard finite cardinalities of $\L^+$-definable subsets of the home sort of $\left(\frak{G},\bbR^\U\right)$. Note that, as an $\L^+$-structure, $\left(\frak{G},\bbR^\U\right)$ is $\aleph_1$-saturated.

Fix a context $x : \bbH$ with $l(x) = 1$, and consider the preschema of partial types $\L^+_x \to \P(\L^+)$ defined by
\[ \psi(x,y) \mapsto \neg I x. \psi(x,y) := \left\{ \log|\psi(x,y)| \geq \left(1-\epsilon\right) \log|x=x| : \epsilon \in \bbQ, 0 < \epsilon < 1\right\}. \]
(This is well-defined because each element of $\bbQ$ is definable without parameters in $\L_0^+$.) In \cite{approxsg}, Hrushovski shows that this map $\bigvee$-defines an ideal of $L^+_x\left(\frak{G},\bbR^\U\right)$. Since $\left(\frak{G},\bbR^\U\right)$ is $\aleph_1$-saturated, by Theorem \ref{V schemas} it is a schema of $\bigvee$-definable ideals for the complete $\L^+$-theory of $\left(\frak{G},\bbR^\U\right)$. The ideal $I = I\left(\frak{G},\bbR^\U\right)$ is the ideal of $\L^+$-definable subsets $X$ of the home sort of coarse pseudo-finite dimension $\mathbf{\delta}(X) < 1$ (see \cite{approxsg} for details). 
\end{example}

\section{Local $I$-stability}
\subsection{The rank $R^I(-,\phi,2)$}

Fix a partitioned $\L$-formula $\phi = \phi(x,y)$ and a complete $\L$-theory $T$. Let $I : \L_x \to \P(\L)$ be a $\bigvee$-definable ideal schema for $T$.

\begin{definition}
Let $\M \models T$ be a model. The relation $R^{I(\M)}(-,\phi,2) \geq n$ is defined on formulas $\psi(x) \in \L_x(M)$ by induction on $n \in \omega$:
\begin{enumerate}[i.]
\item $R^{I(\M)}\left(\psi(x),\phi,2\right) \geq 0$ if, and only if, $[\psi]_{\sim_\M} \notin I(\M)$.
\item $R^{I(\M)}\left(\psi(x),\phi,2\right) \geq n+1$ if, and only if, there is an instance $\phi(x,c) \in \L_x(M)$ of $\phi$ such that \[R^{I(\M)}\left(\psi(x)\wedge \phi(x,c)^t,\phi,2\right) \geq n \] for each $t \in 2$.
\end{enumerate}
\end{definition}

\begin{remark} \label{rank remark}
Let $\psi(x),\theta(x) \in \L_x(M)$. By induction on $n \in \omega$, the following are readily established.
\begin{enumerate}[i.]
\item If $\M \models \forall x\left(\psi(x) \to \theta(x)\right)$, then $R^{I(\M)}\left(\psi(x),\phi,2\right) \geq n$ implies $R^{I(\M)}\left(\theta(x),\phi,2\right) \geq n$. In particular, if $\M \models \forall x\left(\psi(x) \leftrightarrow \theta(x)\right)$, then $R^{I(\M)}\left(\psi(x),\phi,2\right) \geq n$ $\Longleftrightarrow$ $R^{I(\M)}\left(\theta(x),\phi,2\right) \geq n$. 
\item If $R^{I(\M)}\left(\psi(x),\phi,2\right) \geq n$ then $R^{I(\M)}\left(\psi(x),\phi,2\right) \geq m$ for all $m < n$; so $\left\{n \in \omega : R^{I(\M)}\left(\psi(x),\phi,2\right) \geq n\right\}$ forms an initial segment of $\omega$. Therefore the following is well-defined.
\end{enumerate}
\end{remark}

\begin{definition}
For $\psi(x) \in \L_x(M)$, put
\begin{enumerate}[i.]
\item $R^{I(\M)}\left(\psi(x),\phi,2\right) := -\infty$ if $\psi(x) \in I(\M)$,
\item $R^{I(\M)}\left(\psi(x),\phi,2\right) := \infty$ if $R^{I(\M)}\left(\psi(x),\phi,2\right)\geq n$ for every $n \in \omega$,
\item and otherwise \[ R^{I(\M)}\left(\psi(x),\phi,2\right) := \inf \left\{ n \in \omega : R^{I(\M)}\left(\psi(x),\phi,2\right)\not\geq n+1\right\} = \sup \left\{ n \in \omega : R^{I(\M)}\left(\psi(x),\phi,2\right) \geq n \right\}. \]
\end{enumerate}
For a partial type $\pi(x)$, put \[ R^{I(\M)}\left(\pi(x),\phi,2\right) := \inf \left\{ R^{I(\M)}\left(\bigwedge q(x),\phi,2\right) : q(x)\subset \pi(x) \text{ finite}\right\}. \]

$\phi$ is \emph{$I$-stable in $\M$} if 
\[ R^{I(\M)}(x=x,\phi,2) < \infty. \]
$\phi$ is \emph{$I$-stable in $T$} if it is $I$-stable in every model of $T$.
\end{definition}

The rank $R^{I(\M)}(-,\phi,2)$ measures the number of times a definable subset of $\M$ can be iteratively split by instances of $\phi$ over $M$ which are wide and co-wide with respect to $I$. When $I = \left\{\bot\right\}$ is the trivial ideal, we recover Shelah's rank $R(-,\phi,2)$ \cite{ct}. Extending the analogy, we construct appropriate binary trees of collections of formulas whose consistency with $\Th(\M)$ controls the values of $R^{I(\M)}(-,\phi,2)$.

\begin{definition}
For a formula $\psi = \psi(x,z)$ and a nonzero ordinal $\alpha$, $\Tree_\psi[I,\phi,\alpha]$ denotes the following collection of formulas:
\[ \bigcup \left\{ \neg I x.\psi(x,z)\wedge \bigwedge_{l \in K }\phi(x,y_{\sigma\restrict l})^{\sigma[l]} 
: \sigma \in 2^\alpha, K \subset \alpha \text{ finite} \right\}. \]
(So the free variables of $\Tree_\psi[I,\phi,\alpha]$ are among those in $z^\frown (y_s : s \in 2^{<\alpha})$.) When $\alpha = 0$, we set
\[ \Tree_\psi[I,\phi,0] := \neg Ix.\psi(x,z). \]
$\M$ has \emph{wide $\phi$-trees of height $\alpha$ in $\psi(x,c)$} if it realizes $\Tree_{\psi(x,c)}[I,\phi,\alpha]$. Put \[ \Tree[I,\phi,\alpha]:= \Tree_{x=x}[I,\phi,\alpha]. \]
\end{definition}

\begin{proposition} \label{tree is rank}
$\text{}$
\begin{enumerate}[i.]
\item For $n \in \omega$, $\psi(x,c) \in \L_x(M)$, $R^{I(\M)}(\psi(x,c),\phi,2) \geq n \Longleftrightarrow \M \models \Tree_{\psi(x,c)}[I,\phi,n]$.
\item If $\M\models\Tree_{\psi(x,c)}[I,\phi,\omega]$ then $R^{I(\M)}(\psi(x,c),\phi,2) = \infty$. The converse holds if $\M$ is $\omega$-saturated.
\end{enumerate}
\end{proposition}

\begin{proof}
\begin{enumerate}[i.]
\item Straightforward induction on $n \in \omega$.

\item If $\M$ realizes $\Tree_{\psi(x,c)}[I,\phi,\omega]$, then for every $n \in \omega$, $\M$ realizes $\Tree_{\psi(x,c)}[I,\phi,n]$. In view of (i), this is equivalent to $R^{I(\M)}(\psi(x,c),\phi,2)\geq n$ for every $n$, i.e. $R^{I(\M)}(\psi(x,c),\phi,2)=\infty$. Conversely, suppose $R^{I(\M)}(\psi(x,c),\phi,2) \geq n$ for all $n \in \omega$. By (i) again, we have $\M \models \Tree_{\psi(x,c)}[I,\phi,n]$ for every $n \in \omega$. Compactness dictates that 
\[ \Tree_{\psi(x,c)}[I,\phi,\omega] \]
is a type over $c$ with respect to $\M$. If $\M$ is $\omega$-saturated, it must realize this type.
\end{enumerate}
\end{proof}

\begin{proposition} \label{absoluteness}
Let $\M \prec \frak{N}$ be an elementary extension of models of $T$, and let $\pi(x)$ be a partial type.
\begin{enumerate}[i.]
\item $R^{I(\M)}(\pi(x),\phi,2) \leq R^{I(\frak{N})}(\pi(x),\phi,2)$. 
\item If $\M$ is $\omega$-saturated, then $R^{I(\M)}(\pi(x),\phi,2) = R^{I(\frak{N})}(\pi(x),\phi,2)$.
\item $\phi$ is $I$-stable in $T$ if, and only if, $R^{I(\M)}(x=x,\phi,2) < \infty$ for some $\omega$-saturated model $\M \models T$.
\end{enumerate}
\end{proposition}

\begin{proof}
\begin{enumerate}[i.]
\item It suffices to consider the case when $\pi(x)$ consists of a single formula $\psi(x,c)$. A simple induction on $n \in \omega$ using Proposition \ref{tree is rank}.i shows that for all $\psi(x,c) \in \L_x(M)$, if $R^{I(\M)}(\psi(x,c),\phi,2)\geq n$ then $R^{I(\frak{N})}(\psi(x,c),\phi,2) \geq n$, as required.
\item If $R^{I(\frak{N})}(\psi(x,c),\phi,2)\geq n$, then $\frak{N}\models \Tree_{\psi(x,c)}[I,\phi,n]$. If $\M$ is $\omega$-saturated, it also realizes $\Tree_{\psi(x,c)}[I,\phi,n]$, which implies $R^{I(\M)}(\psi(x,c),\phi,2)\geq n$.
\item Suppose $R^{I(\M)}(x=x,\phi,2) = \infty$ for some $\M \models T$. By Proposition \ref{tree is rank}, $\M$ realizes $\Tree[I,\phi,n]$ for every $n \in \omega$. Thus $\Tree[I,\phi,n]$ is a consistent partial type in finitely many variables over a finite set of parameters, so it is realized in every $\omega$-saturated model $\frak{N}$ of $\Th(\M) = T$. By Proposition \ref{tree is rank} again, we have $R^{I(\frak{N})}(x=x,\phi,2) \geq n$ for every $n$, i.e. $R^{I(\frak{N})}(x=x,\phi,2) = \infty$.
\end{enumerate}
\end{proof}

The following makes sense in view of Proposition \ref{absoluteness} and elementary amalgamation.

\begin{definition}
For $\psi(x,c)\in \L_x(M)$, let 
\[ R^I(\psi(x,c),\phi,2) = R^{I(\frak{N})}(\psi(x,c),\phi,2) \]
where $\frak{N}$ is any $\omega$-saturated elementary extension of $\frak{M}$. Similarly, let 
\[ R^I(T,\phi,2) = R^{I(\M)}(x=x,\phi,2), \] 
where $\M$ is any $\omega$-saturated model of $T$.
\end{definition}

In particular, $\phi$ is $I$-stable in $T$ if and only if $R^I(T,\phi,2) < \infty$.
\subsection{The $I$-stable formula theorem for $\bigvee$-definable ideals}

Recall that Shelah's (un)stable formula theorem characterizes local stability in terms of bounds on the number of complete $\phi$-types, the inconsistency of infinite binary $\phi$-trees, definability of complete $\phi$-types over their own domains, and the absence of the order property \cite[Theorem II.2.2]{ct}. Theorem \ref{V-WUFT} extends this result to an approximate setting in the $\bigvee$-definable case. Under this assumption, the condition $R^I(x=x,\phi,2) \geq n$ is type-definable and thereby amenable to applications of compactness. Much of the proof is thus a minor variation of Shelah's, \emph{mutatis mutandis}, but there is a catch: the argument only guarantees definability of wide $\phi$-types over $\omega$-saturated models. We give a detailed proof of this direction to point out where this limitation comes into play; streamlined proofs of the other directions are included for completeness.

\begin{theorem} \label{V-WUFT}
Let $T$ be a complete $\L$-theory, $I$ a schema of $\bigvee$-definable ideals for $T$ in the context $x$, and $\phi = \phi(x,y)$ a partitioned formula. The following are equivalent:
\begin{enumerate}[a.]
\item For every $\lambda \geq 2^{|\L|}$, $\M \models T$, and $A \subset M$, $|A| \leq \lambda$ implies \[ \left| S^{I}_\phi(A, \M)\right| \leq \lambda. \]
\item There exists $\lambda \geq \aleph_0$ such that for every $\M \models T$ and $A \subset M$, $|A| \leq \lambda$ implies \[ \left| S^{I}_\phi(A, \M)\right| \leq \lambda. \]
\item $\Tree[I,\phi,\omega]$ is inconsistent with $T$.
\item $R^I(T,\phi,2)<\infty$.
\item For every $\omega$-saturated model $\M\models T$, every wide $\phi$-type $p \in S^{I}_\phi(\M)$ is $M$-definable.
\end{enumerate}
\end{theorem}

\begin{proof}
(a)$\Rightarrow$(b) is trivial. For (b)$\Rightarrow$(c), suppose $\Tree[I,\phi,\omega]$ is satisfiable, fix $\lambda \geq \aleph_0$, and put
\[ \mu = \inf\left\{\kappa : 2^\kappa > \lambda\right\}; \]
so $\mu \leq \lambda < 2^\mu$. By compactness, the assumption $\neg$(c) implies that $\Tree[I,\phi,\mu]$ is a consistent partial type. Let $(b_s : s \in 2^{<\mu})$ be a realization of $\Tree[I,\phi,\mu]$ in a model $\M$. By definition, the paths $\sigma \in 2^\mu$ determine $I$-wide partial types $p_\sigma := \left\{ \phi(x,b_{\sigma\restrict\alpha})^{\sigma[\alpha]} : \alpha \in \mu \right\}$ which are pairwise (explicitly) contradictory. Taking wide completions obtains the bound $|B| \leq \mu \leq \lambda < 2^\mu = \left| S^I_\phi(B,\M)\right|$ where $B = \bigcup \left\{ b_s : s \in 2^{<\mu} \right\}$. (c)$\Rightarrow$(d) follows immediately from Propositions \ref{tree is rank}.ii and \ref{absoluteness}.ii.

For (d)$\Rightarrow$(e), suppose $R^I(x=x,\phi,2) < \infty$, fix an $\omega$-saturated model $\M$ of $T$, let $p \in S^I_\phi(\M)$ be a wide $\phi$-type over $M$, and put $k = R^I(p,\phi,2)$. Since $\M$ is $\omega$-saturated, $k = R^{I(\M)}(p,\phi,2)$; since $p$ is wide, $k \geq 0$; and by our assumption (d), $k < \omega$. Let $p_0 \subset p$ be a finite subset which satisfies
\[ R^I\left(\bigwedge p_0,\phi,2\right) = k, \]
and let $M_0 \subset M$ be the finite set of parameters which occur in $p_0$. We claim that for every $a \in \M_y$, 
\begin{align*}
\phi(x,a) \in p & \Longleftrightarrow \M \text{ realizes }\Tree_{p_0 \wedge \phi(x,a)}[I,\phi,k] \\
\neg\phi(x,a) \in p & \Longleftrightarrow \M \text{ realizes }\Tree_{p_0 \wedge \neg\phi(x,a)}[I,\phi,k].
\end{align*}
Indeed, if $\phi(x,a) \in p$, then we have
\begin{align*}
k = R^I(p_0,\phi,2) & \geq R^I(\bigwedge p_0\wedge\phi(x,a),\phi,2) \\
                    & \geq R^I(p \cup \left\{\phi(x,a)\right\},\phi,2) = R^I(p,\phi,2) = k.
\end{align*}
By Proposition \ref{tree is rank}.i, it follows that $\M$ realizes $\Tree_{p_0\wedge\phi(x,a)}[I,\phi,k]$. Conversely, suppose $\phi(x,a)\notin p$. As $p$ is a complete $\phi$-type over $M$, this implies $\neg\phi(x,a) \in p$. The argument above shows
\[ R^I(p_0\wedge\neg\phi(x,a),\phi,2) = k. \]
If $\Tree_{p_0\wedge\phi(x,a)}[I,\phi,k]$ were realized in $\M$, i.e. $R^I(p_0\wedge\phi(x,a),\phi,2) \geq k$, then $\phi(x,a)$ gives witness to
\[ R^I(p_0,\phi,2)\geq k+1 \]
which contradicts $R^I(p_0,\phi,2) = k$. The same reasoning applies to $\neg\phi(x,a)$, and we have proved our claim.

Let $z = (y_s : s \in 2^{<k})$ and let $\alpha(y,z)$ be the collection of formulas given by $\Tree_{p_0\wedge \phi(x,y)}[I,\phi,k]$. Consider \[ \tilde{\alpha}(y) := \left\{ \exists z \bigwedge q(y,z) : q(y,z) \subset \alpha(y,z) \text{ finite}\right\}. \]
Because $\M$ is $\omega$-saturated, for every $a \in \M_y$ we have
\[ \M \models \alpha(a,b) \text{ for some } b \in M^{2^{<k}} \Longleftrightarrow \M \models \tilde{\alpha}(a). \]
Thus $\phi(x,a) \in p$ holds if and only if $\M \models \tilde{\alpha}(a)$. Likewise, $\neg\phi(x,a) \in p$ holds if and only if $\M \models \tilde{\beta}(a)$, where $\beta(y,z) := \Tree_{p_0(x)\wedge\neg\phi(x,y)}[I,\phi,k]$ and
\[ \tilde{\beta}(y) := \left\{ \exists z \bigwedge q(y,z) : q(y,z) \subset \beta(y,z) \text{ finite}\right\}. \]
In other words,
\begin{align*}
\left\{ a \in \M_y : \phi(x,a) \in p\right\} & = \tilde{\alpha}(\M), \\
\left\{ a \in \M_y : \neg\phi(x,a) \in p \right\} & = \tilde{\beta}(\M)
\end{align*}
are disjoint subsets of $\M$ which are type-definable over the finite subset $M_0$. As $\M$ is $\omega$-saturated, it follows that $\tilde{\alpha}(y) \cup \tilde{\beta}(y)$ is inconsistent with $\Th(\M) = T$. By compactness, there are finite $\alpha_0 \subset \tilde{\alpha}$ and $\beta_0 \subset \tilde{\beta}$ such that \[ T \models \forall y\left( \bigwedge \alpha_0(y) \to \neg \bigwedge \beta_0(y)\right). \]
Let $d_px.\phi(x,y) \in \L_x(M_0)$ be given by $\bigwedge \alpha_0(y)$. Then we have
\begin{align*}
\phi(x,a) \in p(x) & \Longleftrightarrow \M \models \tilde{\alpha}(a) \Longrightarrow \M \models d_px.\phi(x,a), \\
\neg\phi(x,a) \in p(x) & \Longleftrightarrow \M \models \tilde{\beta}(a) \Longrightarrow \M \models \neg d_p x.\phi(x,a).
\end{align*}
Thus $p$ is $M$-definable, as required.

Finally, assume (e) and fix $\lambda \geq 2^{|\L|}$. Let $\M \models T$ be a model, and let $A \subset M$ be such that $|A| \leq \lambda$. Let $\M_0$ be an elementary substructure of $\M$ which contains $A$ and has cardinality at most $|A| + |\L| \leq |A| \leq \lambda$. $\M_0$ has an $\omega$-saturated elementary extension $\frak{U}$ of cardinality $|U|$ at most $|M_0| + 2^{|\L|} \leq \lambda$. For each $p' \in S^I_\phi(A,\M) = S^I_\phi(A,\frak{U})$, choose a wide completion $p \in S^I_\phi(\frak{U})$; the map \[ S^I_\phi(A,\frak{U}) \to S^I_\phi(\frak{U}),p' \mapsto p \] is injective. By assumption (e), every $p \in S^I_\phi(\frak{U})$ is $U$-definable, and the resulting map $S^I_\phi(\frak{U}) \to \L_y(U)$ given by $p \mapsto d_px.\phi(x,y)$ is also injective. Composing these maps, we obtain the bound 
\[ \left| S^I_\phi(A,\M)\right| = \left| S^I_\phi(A,\frak{P})\right|\leq |\L(U)| \leq |L| + 2^{|\L|} \leq \lambda. \]
This completes the proof.
\end{proof}

In part, Theorem \ref{V-WUFT} characterizes local stability modulo a schema of $\bigvee$-definable ideals in terms of counting wide $\phi$-types. Applying the Erd\H{o}s--Makkai theorem to this characterization informs a natural generalization of Shelah's order property.

\begin{definition}
Let $\M$ be an $\L$-structure. $\M$ \emph{realizes an infinite wide $\phi$-order} (with respect to $I$) if there exist sequences $\b{a} \in \left( \M_x\right)^\omega$, $\b{b} \in \left( \M_y \right)^\omega$ such that for all $i,j \in \omega$ we have $\M \models \phi(a_i,b_j)^{\text{if } i < j}$ and $\tp_\phi\left(a_i / \bigcup\left\{b_i : i \in \omega\right\}\right)$ is $I$-wide.

$\phi$ has the \emph{wide order property} in $T$ with respect to $I$ if there exists a model $\M \models T$ which realizes an infinite $I$-wide $\phi$-order.
\end{definition}

To make this definition amenable to applications of the standard lemma on stretching sequences of indiscernibles, we introduce some notation.

\begin{notation}
For a linear order $X$, $K \subset X$, and $i \in X$, define $H_{i,X} = \left\{\phi(x,y_j)^{\text{if }i < j} : j \in X \right\}$ and 
\begin{align*}
H_{I,\phi,X} & = \bigcup\left\{\neg Ix_i.\bigwedge H_{i,K}(x_i) : i \in \omega, K \subset X \text{ finite}\right\}.
\end{align*}
\end{notation}

\begin{observation}
    \label{orders observation}
\begin{enumerate}[i.]
\item $H_{I,\phi,\omega}$ is satisfiable if and only if $H_{I,\phi,N}$ is satisfiable for all $N \in \omega$, if and only if $H_{I,\phi,X}$ is satisfiable for all linear orders $X$. This follows from compactness and the assumption that $I$ is $\bigvee$-definable.
\item By Corollary \ref{wide types}, $\phi$ does not have the wide order property in $T$ if and only if there exists an $\aleph_1$-saturated model of $T$ which does not realize an infinite wide $\phi$-order.
\item By (i) and (ii), $\phi$ has the wide order property if and only if $\neg\phi$ has the wide order property.
\item Let $A$ be a subset of a model of $T$. If $\phi$ has the wide order property in $T$, then in any $\aleph_0 + |A|^+$-saturated model, there exists an $A$-indiscernible sequence $\b{b} = (b_i : i \in \omega)$ and a sequence $\b{a} = (a_i : i \in \omega)$ such that $\b{a}^\frown\b{b}$ realizes an $I$-wide $\phi$-order. This follows from the standard lemma.
\end{enumerate}
\end{observation}

\begin{theorem} \label{V-WUFT_2}
The following are equivalent:
\begin{enumerate}[a.]
\item $\phi$ is $I$-stable.
\item $\phi$ does not have the wide order property with respect to $I$.
\end{enumerate}
\end{theorem}

\begin{proof}
If (a) fails, there exists a model $\frak{U}$ and a subset $B \subset U$ such that $|S^I_\phi(B)| > \aleph_0 \geq |B|$. For each $p \in S^I_\phi(B)$ consider $X_p := \left\{ b \in B_y : \phi(x,b) \in p \right\}$ and $\mathcal{X} := \left\{ X_p : p \in S^I_\phi(B) \right\}$.
As any two distinct elements of $S^I_\phi(B)$ are explicitly contradictory, we have $|\mathcal{X}| = |S^I_\phi(B)| > \aleph_0 \geq |B|$. The Erd\H{o}s--Makkai theorem delivers, for $i \in \omega$, $p_i \in S^I_\phi(B)$ and $b_i \in B_y$ such that either
\begin{enumerate}[(i)]
\item for all $i,j$, $\phi(x,b_j) \in p_i \Longleftrightarrow i < j$, or 
\item for all $i,j$, $\phi(x,b_j) \notin p_i \Longleftrightarrow i < j$.
\end{enumerate}
Take realizations $a_i \models p_i$ in some sufficiently saturated elementary extension $\frak{W}$. If (i) holds, then the sequences $(a_i : i \in \omega)$, $(b_i : i \in \omega)$ realizes an infinite wide $\phi$-order in $\frak{W}$; if (ii) holds, then these sequences realize an infinite wide $\neg\phi$-order in $\frak{W}$. By Observation \ref{orders observation}.iii, it follows that in $T$, $\phi$ has the wide order property with respect to $I$, as desired.

Now assume $\phi$ has the wide order property in $T$ with respect to $I$. By Observation \ref{orders observation}.i, $H_{I,\phi,X}$ is satisfiable for every linear order $X$. Let $\lambda \geq \aleph_0$ and let $\kappa = \ded(\lambda)$ be its Dedekind number, i.e.
\[ \ded(\lambda) = \sup \left\{ |J| : J \text{ is a linear order with a dense subset of cardinality} \lambda\right\}. \]
Note $\lambda < \kappa$. Let $X$ be a linear order of cardinality $\kappa$ with a dense subset $Q$ of cardinality $\lambda$, and let $\frak{U}\models T$ be a sufficiently saturated model, so that it contains a realization $(b_r : r \in X)$ of $H_{I,\phi,X}$. For $r \in X \setminus Q$, consider the partial type $p_r(x) = \left\{\phi(x,b_q)^{\text{if } r < q} : q \in Q\right\}$ and let $B = \bigcup\left\{b_q : q \in Q\right\}$. Since $(b_r : r \in X)$ realizes $H_{I,\phi,X}$ in $\frak{U}$, each $p_r(x)$ is $I$-wide and thus extends to a complete $I$-wide $\phi$-type over $B$. Moreover, $r \neq r'$ implies $p_r \neq p_{r'}$. We conclude $\left|S^I_\phi(B)\right| \geq \kappa > \lambda \geq |B|$. As $\lambda \geq \aleph_0$ was arbitrary, it follows that $\phi$ is $I$-unstable in $T$ in view of Theorem \ref{V-WUFT}.
\end{proof}

\begin{example}
We work in the first-order language $\L$ of graphs. Let $L_n$ denote the ladder of height $n$ (in which each side is an independent set) and let $K_n$ denote the complete graph on $n$ vertices. Consider $G_n = K_{2^n} \sqcup L_n$. According to {\strokeL}o\'s's theorem, the edge relation $E(x,y)$ is unstable in $(G,\bbR^\U)$. However, $E(x,y)$ is $I$-stable in $(G,\bbR^\U)$, where $I$ is the schema of $\bigvee$-definable ideals defined in Example \ref{delta example}. For a contradiction, assume not. As $I$ is $\bigvee$-definable over a countable set of parameters and $(G,\bbR^\U)$ is $\aleph_1$-saturated, our assumption and Theorem \ref{V-WUFT_2} imply the existence of a sequence $(b_i : i \in \omega)$ in $G$ such that for each $i$, $E_i(x) = \{ E(x,b_j)^{\text{if } i \leq j} : j \in \omega \}$
is $I$-FMP. For each $j$ choose a representing sequence $(b^n_j : n \in \omega) \in \prod_{n \in \omega} G_n$ for $b_j$. For $N > 0$ put $b^n_{<N} = (b^n_j : j < N)$ and $H_N(x_0,\dots,x_{N-1},y_0,\dots,y_{N-1}) := \bigwedge_{i, j < N} E(x_i,y_j)^{\text{if } i \leq j}$.
Let $N \geq 4$.
\begin{claim}
$\{n \in \omega : b^n_{<N} \subset L_n\} \in \U$.
\end{claim}
\begin{proof}[Proof of claim]
Consider the formula
$$\chi(x) := \forall y \left( E(x,y) \to \left(\forall z \left( z \neq x \wedge E(y,z) \to E(x,z)\right)\right)\right).$$
Notice that for each $n$, $G_n \models \chi(v)$ if, and only if, $v \in K_{2^n}$. For a contradiction, assume the claim fails, so that $N_1 := \{ n \in \omega : b^n_{<N} \cap K_{2^n} \neq \emptyset\} \in \U$. By our choice of $b_{<N}$, an application of {\strokeL}o\'s's theorem gives $N_2 := \{n \in \omega : G_n \models \exists \b{x} H_N(\b{x},b^n_{<N})\} \in \U$. Fix $n \in N_1 \cap N_2$, so that $G_n \models \exists \b{x}H_N(\b{x},b^n_{<N})$ and for some $j < N$, $G_n \models \chi(b^n_j)$. But any two vertices in an instance of $H_N(\b{x},\b{y})$ are joined by a path, so in fact for every $j < N$ we must have $G_n \models \chi(b^n_j)$, i.e.
$$N_1 \cap N_2 = \{n \in \omega : G_n \models \exists\b{x}H_N(\b{x},b^n_{<N}) \text{ and } b^n_{<N} \subset K_{2^n} \}.$$
Moreover if $n \in N_1 \cap N_2$ and $G_n \models H_N(a^n_{<N},b^n_{<N})$ then the same reasoning shows $a^n_{<N} \subset K_{2^n}$ as well. But as long as $N \geq 2$, $H_N(\b{x},\b{y})$ will stipulate at least one non-edge requirement between the $a^n_{<N}$ and $b^n_{<N}$; yet $a^n_{<N}$ and $b^n_{<N}$ are contained in $K_{2^n}$. This is absurd, and the proof is complete.
\end{proof}

Now let $n$ be such that 
$$\left(\dag\right) \hspace{0.5cm} G_n \models \exists\b{x}H_N(\b{x},b^n_{<N}) \text{ and } b^n_{<N} \subset L_n.$$
We want to estimate $\log\left|H_N(G_n^N,b^n_{<N})\right|/\log\left|G^N_n\right|$. For each $i < N-1$, $\left|E_i(G_n,b^n)\right| \leq n$ holds whenever $n$ satisfies $\left(\dag\right)$. For $i = N-1$, $\left|E_{N-1}(G_n,b^n_{<N})\right|\leq 2^n$. Thus,
\begin{align*}
    \frac{\log\left|H_N(G_n^N,b^n_{<N})\right|}{\log\left|G_n^N\right|} & \leq \frac{\log(n^{N-1})+ \log(2^n)}
    {\log\left((2^n+2n)^N\right)} \\ & \leq \frac{N-1}{N}\frac{\log n}{\log(2^n + 2n)} + \frac{1}{N} \leq \frac{\log n}{\log(2^n + 2n)} + \frac{1}{4}. \\
\end{align*}
Let $M$ be such that $n \geq M$ implies $\log n/\log(2^n+2n) < \frac{1}{4}$. If we put
\begin{align*}
    N_3 & = \{ n \in \omega : G_n \models \exists\b{x}H_N(\b{x},b^n_{<N}) \text{ and }
    b^n_{<N} \subset L_n\} \cap [M,\omega) \\
    N_4 & = \bigg\{ n \in \omega : (G_n,\bbR) \models \frac{\log\left|H_N(G^N_n,b^n_{<N})\right|}{\log\left|G^N\right|} < \frac{1}{2} \bigg\}
\end{align*}
then what we just showed amounts to $N_3 \subset N_4$. $N_3 \in \U$ follows from the claim above; hence $N_4 \in \U$ holds as well. On the other hand, recalling that $(b_j : j \in \omega)$ witnesses the $I$-wide order property for the edge relation $E(x,y)$ in the structure $(G,\bbR^\U)$, for each $N \in \omega$ we have
$$\st\left( \frac{\log\left|H_N(G^N,b_{<N})\right|}{\log\left|G^N\right|}\right) = 
\frac{1}{N}\delta\left( \prod_{i < N} E_i(G,b_{<N})\right) = 
\frac{1}{N} \sum_{i < N} \delta\left(E_i(G,b_{<N})\right) = 1,$$
by definition of $I$. In particular, $(G,\bbR^\U) \models \log\left|H_N(G^N,b_{<N})\right|/\log\left|G^N\right| \geq \frac{1}{2}$ holds. By {\strokeL}o\'s's theorem, it follows,
$$\left\{ n \in \omega : (G_n,\bbR^\U) \models \frac{\log\left|H_N(G^N_n,b^n_{<N})\right|}{\log\left|G^N\right|} \geq \frac{1}{2} \right\} \in \U.$$
We conclude $\emptyset \in \U$, a contradiction. 
\end{example}
\subsection{The $I$-stable formula theorem fails for $\bigwedge$-definable ideals}

The proof of Theorem \ref{V-WUFT} employs compactness arguments that make critical use of the assumption that $I$ is $\bigvee$-definable. In fact, the statement Theorem \ref{V-WUFT} becomes false in general after substituting ``$\bigwedge$-definable" for ``$\bigvee$-definable," as the following example due to L.C.\ Brown shows.

We only defined the rank $R^I(-,\phi,2)$ with respect to a schema of $\bigvee$-definable ideals, but the same definition works for a schema of $\bigwedge$-definable ideals. Under this modification, the statement ``$R^I(-,\phi,2) \geq \alpha$" is no longer type-definable, but rather $\bigvee$-definable. In the case of $\bigwedge$-definable ideal schema, we instead say that $\M$ \emph{has $I$-wide $\phi$-trees of height $\alpha$ in $\psi$} if there is a tree $\b{b}=(b_s : s \in 2^{<\alpha})$ such that for each finite $K \subset \alpha$ and $\sigma \in 2^\alpha$ there is a finite 
$$q_{\sigma,K}(\b{y}) \subset Ix.\psi(x)\wedge\bigwedge_{k \in K} \phi(x,y_{\sigma\restrict k})^{\sigma[k]}$$
such that $\M \models \neg \bigwedge q_{\sigma,K}(\b{b})$.

Note that the statements of Remark \ref{rank remark} and Proposition \ref{tree is rank}.i hold when ``$\bigwedge$-definable'' is substituted for ``$\bigvee$-definable." Stability modulo $I$ is defined in the same way, too: $\phi(x,y)$ is \emph{$I$-stable} in $T$ if $R^{I(\M)}(x=x,\phi,2) < \infty$ for every model $\M$. However, Proposition \ref{tree is rank}.ii and thus Theorem \ref{V-WUFT} do not hold under this substitution:

\begin{example} \label{type-WUFT counterexample}
(L.~C. Brown)
Let $\L$ be the first-order language of signature $\left\{\leq\right\}$, and let $\M$ be the $\L$-structure given by the disjoint union $\bigsqcup_{n \in \omega\setminus\left\{0\right\}} n$ viewed as a partial order (i.e. with $\leq$ interpreted so that $x \leq y$ if and only if $x,y$ lie in the same $n \in \omega$ and $x \leq y$ in $n$). Let $T$ be the complete $\L$-theory of $\M$, let $D_n(x)$ be the negation of the $\L$-formula that says that $x$ is comparable to exactly $n$ elements, and let $I : \L_x \to \P(\L_x)$ be the schema of $\bigwedge$-definable ideals for $T$ given by
$$\psi(x,y) \mapsto I x.\psi(x,y) := \left\{ \forall x \left( \psi(x,y) \to \bigwedge_{n \in K} D_n(x) \right) : K \subset \omega \text{ finite} \right\}.$$

Consider the $\L$-formula $\theta = \theta(x,y) := x \leq y$. We claim that $R^I(x=x,\theta,2) \geq n$ for every $n \in \omega$, but no model of $T$ has infinite $I$-wide $\theta$-trees. Indeed, fix $n \in \omega$; let $l = |2^{<n}| + 1 = 2^n$. Consider the ordering $\prec$ on $2^{< n}$ defined as follows: let $\eta \neq \nu \in 2^{< n}$.
\begin{itemize}
\item If $\eta \bot \nu$, then $\eta \prec \nu$ if and only if $\eta[\left|\eta\wedge\nu\right|] = 1$;
\item Else, without loss of generality, $\eta \subsetneq \nu$; say $\eta \prec \nu$ if and only if $\nu[\left|\eta\wedge\nu\right|] = 0$.
\end{itemize}
Then $\prec$ defines a linear order on $2^{<n}$; let $b : 2^{<n} \to l, s \mapsto b_s$ be the unique order-preserving embedding which maps $2^{<n}$ onto an initial segment of $l$. For each $\sigma \in 2^n$, observe that $\theta_\sigma(x,\b{b}) := \bigwedge_{i < n} \theta(x,b_{\sigma \upharpoonright i})^{\sigma[i]}$ has nontrivial intersection with the finite part of $\frak{U}$: $\frak{U} \models \theta_{b^{-1}(k)}(k,\b{b})$ for $k < n - 1$, and $\frak{U} \models \theta_{0\cdots 0}(2^n - 1)$. So $\b{b}$ realizes a $I$-wide $\theta$-tree, by definition of $I$. It follows that $R^I(x=x,\theta,2) \geq n$.

Assume for a contradiction that some model $\frak{U}$ of $T$ has a $I$-wide $\theta$-tree $(b_s : s \in 2^{<\omega})$ of height $\omega$. Let $\sigma \in 2^\omega$ be the sequence which is constantly equal to $1$; by assumption, there is an $a \in \frak{U}$ such that $\frak{U} \models a \leq b_{\sigma(n)}$ for all $n \in \omega$ and $\tp(a/\b{b})$ is $I$-wide. The fact that $\tp(a/\b{b})$ is $I$-wide implies that $a$ lies in the finite part of $\frak{U}$, but the fact that $a \leq b_{\sigma(n)}$ for all $n \in \omega$ implies that $a$ is comparable to infinitely many elements, a contradiction.
\end{example}
\subsection{WIP and WSOP}

Recall that a formula has the order property if and only if it has the independence property or some Boolean combination of instances of that formula has the strict order property (\cite{ct}, Theorem II.4.7). In this section we remark that the same holds in an approximate setting.

Fix a schema $I$ of $\bigvee$-definable ideals for $T$. Let $\phi = \phi(x,y)$ be a partitioned $\L$-formula.

\begin{definition}
$\phi$ has the \emph{wide independence property} with respect to $I$ in $\M$ if there is a sequence $\b{b} \in \M_y^\omega$ such that for each $\sigma \in 2^\omega$,
\[ \left\{ \phi(x,b_i)^{\sigma[i]} : i \in \omega \right\} \]
is $I$-wide. $\phi$ has the \emph{wide independence property} in $T$ (with respect to $I$) if it has the wide independence property in some model of $T$.
\end{definition}

\begin{observation}
By compactness, $\phi$ has the wide independence property if, and only if, for every $N \in \omega$,
\[ \bigcup_{\sigma \in 2^N} \neg I x . \bigwedge_{i < N} \phi(x,y_i)^{\sigma[i]} \]
is a type. In particular, $\phi$ has the wide independence property if and only if $\phi$ has the wide independence property in every (equivalently, some) $\omega$-saturated model of $T$.
\end{observation}

\begin{definition}
$\phi$ has the \emph{wide strict order property} (with respect to $I$, in $T$) if there is a model $\M \models T$ and a sequence $\b{b} \in \M_y^\omega$ such that for all $i, j \in \omega$, if $i < j$, then $\neg \phi(x,b_i) \wedge \phi(x,b_j)$ is $I$-wide whereas $\phi(x,b_i) \wedge \neg \phi(x,b_j)$ is $I$-thin. 
\end{definition}

\begin{theorem}
    \label{WOP=WIPvWSOP}
If $\phi$ has the wide order property, then it has the wide independence property or some Boolean combination of instances of $\phi$ has the wide strict order property.
\end{theorem}

The proof of Theorem \ref{WOP=WIPvWSOP} follows that of \cite[Theorem 2.67]{simon}; the only difference is that we must check that wideness and thinness of two particular formulas are preserved upon replacing the original indiscernible sequence which witnesses the wide order property to one that is stretched to order type $\bbQ$.

\begin{proof}
Suppose $\phi$ has the wide order property but not the wide independence property. Fix an $\omega$-saturated model $\frak{U}$. Then there exists an indiscernible sequence $\b{b} \in \frak{U}_y^\omega$ such that for each $i \in \omega$, $\left\{\phi(x,b_j)^{\text{if } i < j} : j \in \omega \right\}$ is $I$-wide. Since $\phi$ does not have the wide independence property, there exists $n \in \omega$ and $\eta \in 2^n$ such that $\bigwedge_{i < n} \phi(x,b_i)^{\eta[i]}$ is $I$-thin. 

Edit the formula $\bigwedge_{i < n} \phi(x,b_i)^{\eta[i]}$ as follows: while there exists some subformula of the form $\neg\phi(x,b_i) \wedge \phi(x,b_{i+1})$, change the rightmost such instance to $\phi(x,b_i)\wedge\neg\phi(x,b_{i+1})$; if no such subformulae are present, halt and return the current formula. This algorithm terminates and returns a formula of the form $\bigwedge_{i < N} \phi(x,b_i) \wedge \bigwedge_{N \leq i < n} \neg\phi(x,b_i)$.
By our choice of $\b{b}$, this formula is wide; so there is a step in this procedure where we pass from a thin formula to a wide one. I.e., there exist $i_0 < n$ and $\gamma \in 2^n$ such that $\theta(x,\b{b}) \wedge \neg \phi(x,b_{i_0}) \wedge \phi(x,b_{i_0+1})$
is thin, whereas $\theta(x,\b{b})\wedge \phi(x,b_{i_0}) \wedge \neg\phi(x,b_{i_0+1})$ is wide, where we define
\[ \theta(x,\b{b}) := \bigwedge_{i < i_0} \phi(x,b_i)^{\gamma[i]} \wedge \bigwedge_{i_0 + 1 < i < n} \phi(x,b_i)^{\gamma[i]}. \]
Let $q(\b{y},y_{i_0},y_{i_0+1}) \subset \neg Ix. \theta(x,\b{y}) \wedge \neg \phi(x,y_{i_0}) \wedge \phi(x,y_{i_0+1})$ witness thinness of the first formula. Stretch $\b{b}$ to an indiscernible sequence $\b{c} = (c_i : i \in \bbQ)$ via the standard lemma, and write $\b{c} := (c_i : i < i_0) \frown (c_i : i_0+1 < i < n)$. Let $J = \bbQ \cap (i_0, i_0+1)$ and $\psi(x,y) := \theta(x,\b{c}) \wedge \phi(x,y)$. Since $\frak{U} \models \neg\bigwedge q(\b{b},b_{i_0},b_{i_0+1})$ and $\b{b}$ is indiscernible, we have $\neg\bigwedge q(\b{y},y_{i_0},y_{i_0+1}) \in \EM_{l(y)\cdot n}(\b{b})$. But $\b{c}$ realizes the Ehrenfeucht--Mostowski type of $\b{b}$, and so 
\[ \frak{U} \models \neg\bigwedge q(c_{j_0},\dots,\hat{c_{j_{i_0}}},\hat{c_{j_{i_0+1}}},\dots,c_{j_{n-1}},c_{j_{i_0}},c_{j_{i_0+1}}) \]
holds whenever $j_0 < \cdots < j_{n-1} \in \bbQ$. In particular, if $i < j \in (i_0,i_0+1)_\bbQ$, we have $\frak{U}\models\neg\bigwedge q(\b{c},c_i,c_j)$, i.e. $\neg\psi(x,c_i) \wedge \psi(x,c_j) \in I$. Following similar reasoning, we see that for $i < j \in (i_0,i_0+1)_\bbQ$, 
\[ \frak{U} \models \neg Ix.\theta(x,\b{c}) \wedge \phi(x,c_i) \wedge \neg \phi(x,c_j) \]
holds, i.e. $\psi(x,c_i) \wedge \neg\psi(x,c_j) \notin I$. Let $K$ be the image of an embedding of $\omega^*$ ($\omega$ with the reverse ordering) in $(i_0,i_0+1)$. Then $(c_k : k \in K)$ witnesses the wide strict order property for $\psi(x,y)$. This completes the proof. 
\end{proof}

\section{Global $I$-stability}
\subsection{Products of $\bigvee$-definable ideals}

Fix a complete $\L$-theory $T$ and a model $\frak{U} \models T$. Let $I \subset L_x(\frak{U})$, $J \subset L_y(\frak{U})$ be ideals.

\begin{definition}
The product $I \otimes J$ consists of all definable subsets $D \subset \frak{U}_{x^\frown y}$ for which there exists a (definable) $F \in I$ such that $\left\{ a \in \frak{U}_x : D_a \notin J \right\} \subset F$.
\end{definition}

The following holds with no additional constraints on $I$ and $J$.

\begin{proposition}
$I \otimes J \subset L_{x \frown y}(\frak{U})$ is an ideal.
\end{proposition}

\begin{proof}
If $F_0$, $F_1 \in I$ witness $D_0$, $D_1 \in I \otimes J$ respectively, then $F_0 \cup F_1 \in I$ witnesses $D_0 \cup D_1 \in I \otimes J$. If $F$ witnesses $D \in I \otimes J$ and $E \subset D$, then $F$ also witnesses $E \in I \otimes J$. Finally $\emptyset \in I \otimes J$ is witnessed by $\emptyset \in I$.
\end{proof}

\begin{lemma} \label{product props}
(\cite{approxeq}, p. 57) Suppose $\frak{U}$ is $\kappa$-saturated and strongly homogeneous and $A \subset U$ is such that $|A|< \kappa$. Assume that $J$ is $\Aut(\frak{U})$-invariant.
\begin{enumerate}[i.]
\item If $\tp(a/A)$ is $I$-wide and $\tp(b/Aa)$ is $J$-wide then $\tp(ab/A)$ is $I\otimes J$-wide.
\item If $\tp(ab/A)$ is $I\otimes J$-wide, then $\tp(a/A)$ is $I$-wide.
\item Further assume that $J$ is $\emptyset$-$\bigvee$-definable. Then $\tp(ab/A)$ is $I\otimes J$-wide if, and only if, $\tp(a/A)$ is $I$-wide and $\tp(b/Aa)$ is $J$-wide.
\end{enumerate}
\end{lemma}

\begin{proof}
For (i), suppose $\tp(ab/A)$ is $I\otimes J$-thin; then there is an $A$-definable set $X \subset \frak{U}_{x^\frown y}$ and a definable $F \in I$ such that $ab \in X$ and $\left\{c \in \frak{U}_x  : X_c \notin J\right\} \subset F$. We want to show $\tp(a/A)$ is $I$-thin or $\tp(b/Aa)$ is $J$-thin; so suppose $\tp(b/Aa)$ is $J$-wide. Then $b \in X_a \notin J$, which implies $a \in F$ by definition. Since $J$ is invariant (over $\emptyset$) and $\frak{U}$ is sufficiently strongly homogeneous, it follows that $\tp(a/A) \vdash x \in F$, i.e. $\left\{ a' \in \frak{U}_x : \tp(a'/A) = \tp(a/A) \right\} \subset F$. Compactness and sufficient saturation thus deliver an $A$-definable set $G \ni a$ contained in $F$. In particular, $G \in I$ and $G \in \tp(a/A)$, so $\tp(a/A)$ is $I$-thin.

For (ii), suppose $\tp(a/A)$ is $I$-thin. Then $a \in X$ for some $I$-thin, $A$-definable set $X$. Putting $Y = \left\{ cd \in \frak{U}_{x^\frown y} : c \in X\right\}$, we obtain $\left\{c \in \frak{U}_x  : Y_x \notin J\right\} \subset X \in I$. Since $Y$ is $A$-definable and $ab \in Y$, it follows that $\tp(ab/A)$ is $I\otimes J$-thin.

For (iii), assume that $\tp(b/Aa)$ is $J$-thin. Let $\psi(x,y) \in \L(A)$ be such that $\frak{U}\models \psi(a,b)$ and $\psi(a,\frak{U}) \in J$. Because $J$ is $\bigvee$-definable, there is a formula $\theta(x) \in \L$ such that $\psi(c,\frak{U}) \in J$ if, and only if, $\frak{U}\models \theta(c)$. Consider the set $Y \subset \frak{U}_{x^\frown y}$ defined by $\psi(x,y)\wedge\theta(x)$. Then $Y$ is $A$-definable and $ab \in Y$. Moreover, we have $\left\{c \in \frak{U}_x  : Y_c \notin J\right\} = \emptyset$ because for any $c$, if $\frak{U}\models \theta(c)$, then $Y_c = \psi(c,\frak{U}) \in J$ by our choice of $\theta(x)$, and if $\frak{U}\models \neg\theta(c)$ then $Y_c = \emptyset \in J$. So $Y \in I\otimes J$, and we conclude that $\tp(ab/A)$ is $I\otimes J$-thin, as desired.
\end{proof}

\begin{lemma} \label{V product is V}
Suppose $\frak{U}$ is $\omega$-saturated. If $I$, $J$ are $\emptyset$-$\bigvee$-definable ideals, then so is $I \otimes J$.
\end{lemma}
\begin{proof}
It suffices to prove the following claim: (the class of) $\psi(x,y,c) \in L_{x^\frown y}(\frak{U})$ is in $I \otimes J$ if and only if there exists a formula $\theta(x,w) \in \L$, a finite $p(w) \subset \neg Ix.\theta(x,w)$, and a finite $q(x,z) \subset \neg Jy.\psi(x,y,z)$ such that
\[ \frak{U} \models \exists w \left( \neg\bigwedge p(w) \wedge \forall x \left( \bigwedge q(x,c) \to \theta(x,w)\right) \right). \]
For the forward direction, suppose $\psi(x,y,c) \in I \otimes J$. Then by definition there exists $\theta(x,d) \in I$ such that for all $a \in \frak{U}_x $, if $\psi(a,y,c)\notin J$, then $\frak{U} \models\theta(a,d)$. That is, $\neg Jy.\psi(x,y,c) \vdash \theta(x,d)$ holds. By compactness and saturation, there exists a finite $q(x,z) \subset \neg Jy.\psi(x,y,z)$ such that $q(x,c) \vdash \theta(x,d)$. Moreover, as $\theta(x,d) \in I$, there is $p(w) \subset \neg Ix.\theta(x,w)$ such that $\frak{U} \models \neg \bigwedge p(d)$. Thus there exist $\theta(x,w)$, $p(w)$, and $q(x,z)$ such that
\[ \frak{U} \models \neg\bigwedge p(d) \wedge \forall x \left(\bigwedge q(x,c) \to \theta(x,d) \right), \]
as desired. The other direction is trivial.
\end{proof}

\begin{corollary} \label{associativity}
Let $\frak{U}$ be $\omega$-saturated and strongly homogeneous. If $I$, $J$, $K$ are $\emptyset$-$\bigvee$-definable ideals, then $I \otimes (J \otimes K) = (I \otimes J) \otimes K$.
\end{corollary}

\begin{proof}
Let $D \in L_{x^\frown y^\frown z}(\frak{U})$, and let $A \subset U$ be a finite set over which $D$ is definable. Suppose $D$ is $I \otimes (J \otimes K)$-wide. Then, by saturation, there exists $abc \in D$ such that $\tp(abc/A)$ is $I \otimes (J \otimes K)$-wide. Note that by Lemma \ref{V product is V}, each of $I \otimes J$, $J \otimes K$, $I \otimes (J \otimes K)$, and $(I \otimes J) \otimes K$ are $\emptyset$-$\bigvee$-definable. Hence we may apply Lemma \ref{product props}.iii:
\begin{align*}
\tp(abc/A) \text{ is } I\otimes(J\otimes K)\text{-wide} & \Longleftrightarrow \tp(a/A) \text{ is } I\text{-wide and } bc/Aa \text{ is } J\otimes K\text{-wide} \\
& \Longleftrightarrow \tp(a/A) \text{ is } I\text{-wide, } tp(b/Aa) \text{ is } J\text{-wide, and } \tp(c/Aab) \text{ is } K\text{-wide} \\
& \Longleftrightarrow \tp(ab/A) \text{ is } I\otimes J\text{-wide and } \tp(c/Aab) \text{ is } K\text{-wide} \\
& \Longleftrightarrow \tp(abc/A) \text{ is } (I\otimes J) \otimes K\text{-wide}.
\end{align*}
Since $D$ is $A$-definable and $abc \in D$, the conclusion that $\tp(abc/A)$ is $(I \otimes J)\otimes K$-wide implies $D$ is $(I\otimes J)\otimes K$-wide, as required. The other direction is the same.
\end{proof}

\begin{definition}
Let $\frak{U}$ be $\omega$-saturated and strongly homogeneous. $I \subset L_1(\frak{U})$ be an $\emptyset$-$\bigvee$-definable ideal in one free variable. The ideals $I^{\otimes n} \subset L_n(\frak{U})$ are defined by induction on $n \in \omega \setminus \left\{0\right\}$: $I^1 = I$, $I^{\otimes(n+1)} = I^{\otimes n} \otimes I$.
\end{definition}

\begin{remark}
$\text{}$
\begin{enumerate}[i.]
\item By Lemma \ref{V product is V}, $I^{\otimes n}$ is an $\emptyset$-$\bigvee$-definable ideal of $\frak{U}$, and thus determines a schema $\mathbf{I}_n$ of $\bigvee$-definable ideals for $T$. Moreover, if $\frak{W} \models T$ is another $\omega$-saturated and strongly homogeneous model, then for all $n$, we have $\mathbf{I}_n(\frak{W}) = \left( \mathbf{I}_1(\frak{W}) \right)^{\otimes n}$. I.e., if $\mathbf{I}$ is a $\bigvee$-definable ideal schema for $T$ and $\frak{U}$, $\frak{W}\models T$ are $\omega$-saturated and strongly homogeneous, then $\mathbf{I}(\frak{U})^{\otimes n}$ and $\mathbf{I}(\frak{W})^{\otimes n}$ determine the same schema of $\bigvee$-definable ideals for $T$ in $n$ free variables, up to equivalence. Accordingly, by a slight abuse of notation, if $l(x) = n$, we let 
\[ \L_x \to \P(\L), \text{ } \psi(x,y) \mapsto \neg \mathbf{I} x.\psi(x,y) \]
denote any schema of $\bigvee$-definable ideal for $T$ which defines $I^{\otimes n}$.
\item Corollary \ref{associativity} implies $I^{\otimes n} = I^{\otimes l} \otimes I^{\otimes k}$ whenever $l$, $k$ are positive integers whose sum is $n$. Thus, if $\frak{U}$ is $|A|^+$-saturated and strongly homogeneous, then $\tp(ab/A)$ is $I^{\otimes n}$-wide if and only if $\tp(a/A)$ is $I^{\otimes l}$-wide and $\tp(b/Aa)$ is $I^{\otimes k}$-wide for every $a \in U^l$ and $b \in U^k$. 
\end{enumerate}
\end{remark}
\subsection{Counting wide complete types}

Let $T$ be a complete theory, and let $I$ be a schema of $\bigvee$-definable ideals for $T$ in one free variable. To ease notation, we let $S^I_x(A) := S^{I^{\otimes l(x)}}_x(A)$. Likewise, if $\phi \in \L_x$, we write $S^I_\phi(A) := S^{I^{\otimes l(x)}}_\phi(A)$.

\begin{theorem} \label{global wide stability}
The following are equivalent:
\begin{enumerate}[a.]
\item Every partitioned formula $\phi = \phi(x,y)$ is $I^{\otimes l(x)}$-stable.
\item Every partitioned formula $\phi = \phi(x,y)$ with $l(x) = 1$ is $I$-stable.
\item For every $\lambda = \lambda^{|T|}$, $\M \models T$, and $A \subset M$, if $|A| \leq \lambda$ then $\left|S^I_1(A,\M)\right| \leq \lambda$.
\item For every $\lambda = \lambda^{|T|}$, $\M \models T$, $A \subset M$, and $n \in \omega \setminus\left\{0\right\}$, if $|A| \leq \lambda$, then $\left|S^I_n(A,\M)\right|\leq\lambda$.
\end{enumerate}
\end{theorem}

\begin{definition}
$T$ is \emph{$I$-stable} if the equivalent conditions of Theorem \ref{global wide stability} hold. Otherwise, $T$ is \emph{$I$-unstable}.
\end{definition}

The proof of Theorem \ref{global wide stability} relies on the following

\begin{lemma} \label{res lemma}
Consider the restriction $\res : S_{n+1}(A) \to S_n(A)$, $\tp(ab/A) \mapsto \tp(a/A)$.
\begin{enumerate}[i.]
\item $\res$ maps $S^I_{n+1}(A)$ onto $S^I_n(A)$.
\item If $\tp(a/A)$ is $I^{\otimes n}$-wide, then we have a bijection
\begin{align*}
S_1^I(Aa) & \to \res^{-1}(\tp(a/A)) \cap S^I_{n+1}(A), \\
\tp(b/Aa) & \mapsto \tp(ab/A)
\end{align*}
\end{enumerate}
\end{lemma}
\begin{proof}
We work in a sufficiently (e.g. $\aleph_0+|A|^+$-) saturated and strongly homogeneous model $\frak{U}$. 

If $\tp(ab/A)$ is $I^{\otimes n+1}$-wide, then $\tp(a/A)$ is $I^{\otimes n}$-wide by Proposition \ref{product props}. On the other hand, if $\tp(a/A)$ is $I^{\otimes n}$-wide, take $b$ so that $\tp(b/Aa)$ is $I$-wide; then by Proposition \ref{product props} again, $\tp(ab/A)$ is $I^{\otimes n+1}$-wide. This proves the first item.

For the second, Proposition \ref{product props} guarantees the map in question is well-defined, and it is clearly injective. For surjectivity, suppose $\tp(a'b'/A) \in \res^{-1}(\tp(a/A))\cap S^I_{n+1}(A)$. Then, in particular, $a$ and $a'$ have the same type over $A$, so by strong homogeneity there is an automorphism $\sigma \in \Aut(\frak{U}/A)$ mapping $a'$ to $a$. Putting $b = \sigma(b')$ obtains $\tp(a'b'/A) = \tp(ab/A)$. Now, since $\tp(a'b'/A)$ is $I^{\otimes n+1}$-wide, Proposition \ref{product props} guarantees $\tp(b'/Aa')$ is $I$-wide. It follows that $\tp(b/Aa) = \sigma(\tp(b'/Aa'))$ is $I$-wide, as required.
\end{proof}

\begin{proof}[Proof of Theorem \ref{global wide stability}] (a)$\Rightarrow$(b) is trivial. For (b)$\Rightarrow$(c), fix $\lambda = \lambda^{|T|}$, let $l(x) = 1$, and consider the map
\[ S^I_1(A) \to \prod \left\{ S^I_\phi(A) : \phi \in \L_x\right\}, \hspace{0.25cm} p \mapsto (p \restrict \phi : \phi \in \L_x). \]
This map is well-defined because if $p \in S_1(A)$ is $I$-wide, then so are all of its restrictions to $\phi$-types; it is clearly injective. We assumed that every $\phi = \phi(x,y)$ with $l(x) = 1$ is $I$-stable. Noting that $\lambda \geq 2^{|T|}$, it follows by Theorem \ref{V-WUFT} that if $|A| \leq \lambda$, then $|S^I_\phi(A)|\leq \lambda$. Therefore we conclude $\left|S^I_1(A)\right| \leq \lambda^{|T|} = \lambda$.

For (c)$\Rightarrow$(d), the proof is by induction on $n \in \omega\setminus\{0\}$. The base case $n = 1$ is our assumption (c). Assume the result holds for all $0 < k \leq n$. Let $\lambda = \lambda^{|T|}$ and fix $A$ with $|A|\leq \lambda$. We want to show $\left| S^I_{n+1}(A) \right| \leq \lambda$. Lemma \ref{res lemma}.i implies that the restriction map $\res : \tp(ab/A) \mapsto a/A$ restricts to a surjection $S^I_{n+1} \to S^I_n(A)$ and so we have $S^I_{n+1}(A) = \bigcup\left\{ \res^{-1}(p) : p \in S^I_n(A) \right\}$. Applying the inductive hypothesis $k = n$ to $\lambda$ and $A$ gives $|S^I_n(A)|\leq \lambda$. It thus suffices to show $|\res^{-1}(p)|\leq \lambda$ for each $p \in S^I_n(A)$. Fix such a $p$ and let $a$ be a realization. By Lemma \ref{res lemma}.ii, it follows, $|\res^{-1}(p)| = \left|S^I_1(Aa)\right|$. The inductive hypothesis $k = 1$ applied to $\lambda$ and $Aa$ implies $\left|S^I_1(Aa)\right| \leq \lambda$, as required.

For (d)$\Rightarrow$(a), let $\phi = \phi(x,y)$ be a partitioned formula. In view of Theorem \ref{V-WUFT} and Lemma \ref{V product is V}, it suffices to show that there exists $\lambda \geq 2^{|T|}$ such that $\left|S^I_\phi(A)\right|\leq \lambda$ whenever $|A|\leq\lambda$. Fix any $\lambda = \lambda^{|T|}$. Fix $A$ with $|A|\leq \lambda$. For each $p \in S^I_\phi(A)$, choose a wide completion $\tilde{p} \in S^I_x(A)$. The assignment $p \mapsto \tilde{p}$ is an injection $S^I_\phi(A) \to S^I_x(A)$; the assumption (d) implies $\left|S^I_\phi(A)\right|\leq \lambda$, completing the proof.
\end{proof}

\section{Independence}
Following Shelah \cite{ct}, it would have been convenient to fix a ``monster model'' $\frak{U}$ of $T$ and an ideal $I(\frak{U}) \subset L_x(\frak{U})$ and define when $T$ is approximately stable directly in terms of $I(\frak{U})$. The issue is that $I(\frak{U})$ prima facie depends on the choice of $\frak{U}$, and in adopting this approach one must show that
\begin{enumerate}[i.]
\item under certain constraints on its descriptive complexity, $I(\frak{U})$ corresponds to a functor $I$ from $\Mod(T)$ to the category of Boolean algebras with a distinguished ideal, and
\item whether $T$ is approximately stable modulo $I$ does not depend on the choice of $\frak{U}$, as long as it enjoys adequate saturation (and strong homogeneity and universality).
\end{enumerate}
This is backwards: it seems cleaner to instead make rigorous the notion of a schema of ideals and then deduce (i) and (ii) as a posteriori features of the definition (cf. Theorems \ref{V schemas} and \ref{absoluteness}). In the author's estimation, it was worth foregoing convenience in exchange for better technical footing up to this point; but in what follows, this tradeoff no longer makes sense. In this connection we adopt a standard proviso, warranted by Theorem \ref{V schemas}:

\begin{proviso}
$\frak{U}$ denotes a $\kappa$-saturated and strongly homogeneous model of a fixed complete theory $T$, where $\kappa$ is ``sufficiently large.'' If $I$ is a schema of $\bigvee$-definable ideals for $T$, then we identify it with its $\frak{U}$-points. $A$, $B$, $C$, etc. denote \emph{small} subsets of $U$, i.e. those of cardinality less than $\kappa$. Unless explicitly stated otherwise, every set of formulas $\pi(x)$, $\tau(x)$ under consideration is defined over a small subset. We write $\models \pi(a)$ if $\frak{U} \models \bigwedge q(a)$ for all finite $q(x) \subset \pi(x)$, and $\models \pi(x)$ if $\models \pi(a)$ for some $a \in \frak{U}_x$; so by compactness and saturation, $\models \pi(x)$ holds if and only if $\pi(x)$ is a type. Likewise, we write $\pi(x) \vdash \tau(x)$ if $\pi(\frak{U}) \subset \tau(\frak{U})$. By compactness and saturation, this is equivalent to the condition that for every finite $\tau_0(x) \subset \tau(x)$, there exists a finite $\pi_0(x) \subset \pi(x)$ such that $\models \forall x \left(\bigwedge \pi_0(x) \to \bigwedge \tau_0(x)\right)$.
For two tuples $\b{a}$, $\b{a}'$ in the same order type, we write $\b{a} \equiv_A \b{a}'$ if $\tp(\b{a}/A) = \tp(\b{a}'/A)$---equivalently, by strong homogeneity, if $\sigma(\b{a}) = \b{a}'$ for some $\sigma \in \Aut(\frak{U}/A)$.
\end{proviso}
\subsection{$I$-forking}

Let $I$ be a schema of $\bigvee$-definable ideals for a complete $\L$-theory $T$ in one free variable.

\begin{definition}
A formula  $\psi(x,c) \in \L_x(U)$ \emph{$k$-$I$-divides over $A$} if there is a sequence $(c_i : i \in \omega)$ in $\tp(c/A)$ and a finite 
\[ q(y_0,\dots,y_{k-1}) \subset \neg Ix.\bigwedge_{i < k} \psi(x,y_i) \]
such that $\models \neg \bigwedge q(c_{l_0},\dots,c_{l_{k-1}})\footnote{Note that the ordering of the $c_{l_i}$'s is irrelevant.}$ holds for all $K = \left\{l_i : i < k\right\} \in { \omega \choose k}$. In this situation we say that $\psi(x,c)$ \emph{affords the witnesses} $q(y_0,\dots,y_{k-1})$, $(c_i : i \in \omega)$ \emph{for $k$-$I$-dividing over $A$}. 

$\psi(x,c)$ \emph{$I$-divides over $A$} if there is an $A$-indiscernible sequence $(c_i : i \in \omega)$ in $\tp(c/A)$ such that $\left\{\psi(x,c_i) : i \in \omega\right\}$ is $I$-thin. In this case we say $\psi(x,c)$ \emph{affords the witness} $(c_i : i \in \omega)$ \emph{for $I$-dividing over $A$}. A partial type $\pi(x)$ \emph{$I$-divides over $A$} if it implies a formula which $I$-divides over $A$.
\end{definition}

Note that it only makes sense to discuss non-$I$-dividing for wide formulas and types: if $\psi(x,c)$ is thin, then it $1$-$I$-divides over $\left\{c\right\}$.

The next few results establish some basic properties of $I$-dividing in line with standard facts for dividing, replacing ``consistent'' with ``wide." A number of the proofs similarly generalize in a straightforward manner; we omit proofs in such cases and refer the reader to Section 7.1 of \cite{tz} for more details. 

\begin{proposition} \label{I-dividing basics}
\begin{enumerate}[i.]
\item $\psi(x,c)$ $I$-divides over $A$ if and only if it $k$-$I$-divides over $A$ for some $k$.
\item If $\phi(x,b) \vdash \psi(x,c)$ and $\psi(x,c)$ $k$-$I$-divides over $A$, then $\phi(x,b)$ $k$-$I$-divides over $A$. In particular, the partial type $\pi(x)$ $I$-divides over $A$ if and only if $\bigwedge \pi_0(x)$ $I$-divides over $A$ for some finite $\pi_0(x) \subset \pi(x)$.
\item Let $A \subset B$. If $\phi(x,b)$ $I$-divides over $B$ then $\phi(x,b)$ $I$-divides over $A$. Conversely, if $\phi(x,b)$ $I$-divides over $A$, there is an $A$-conjugate $B' \equiv_A B$ such that $\phi(x,b)$ $I$-divides over $B'$. In fact, if $\phi(x,b)$ affords the witnesses $q(y_0,\dots,y_{k-1})$ and $\b{b}$ for $k$-$I$-dividing over $A$ such that $\b{b}$ is $A$-indiscernible, it affords the same witnesses for $k$-$I$-dividing over $B'$.
\item $\pi(x)$ $I$-divides over $A$ if and only if $\pi(x)$ $I$-divides over some model $M \supset A$ if and only if $\pi(x)$ $I$-divides over $\acl(A)$.
\item Let $p(x) \in S^I(\frak{U})$ be a wide global type. If $p(x)$ is $A$-invariant, then it does not $I$-divide over $A$. In particular, the conclusion holds if $p(x)$ is $A$-definable.
\end{enumerate}
\end{proposition}

\begin{proof}
The proof of (i) is a straightforward generalization of the parallel fact for dividing (see \cite{tz}), once the correct definitions are in place: the requirement of a uniform witness $q(\overline{y})\subset \neg Ix.\bigwedge_{i<k}\psi(x,y_i)$ to $k$-$I$-dividing over $A$ along a given, but not necessarily $A$-indiscernible, sequence $\b{c}$ in $\tp(c/A)$ forces any realization $\b{c}'$ of the Ehrenfeucht--Mostowski type of $\b{c}$ to satisfy $\neg \bigwedge q(\overline{y})$ along any subsequence of length $k$, which is necessary and sufficient for $\psi(x,y)$ being $k$-$I$-thin along $\b{c}'$ if it is additionally $A$-indiscernible.

For (ii), suppose $\psi(x,c)$ affords $(c'_i : i \in \omega)$ and $q(y_0,\dots,y_{k-1}) \subset \neg Ix.\bigwedge_{i<k}\psi(x,y_i)$ for $k$-$I$-dividing over $A$. Let $\lambda = \beth_k^+(|T|)$. By compactness, there exists a sequence $(c_i : i < \lambda)$ in $\tp(c/A)$ which, along with the same $q(y_0,\dots,y_{k-1})$, is afforded by $\psi(x,c)$ for $k$-$I$-dividing over $A$. For each $i$, there is $\sigma_i \in \Aut(\frak{U}/A)$ such that $\sigma_i(c) = c_i$; put $b_i := \sigma_i(b)$. Then, as $\phi(x,b) \vdash \psi(x,c)$, we have $\phi(x,b_i) \vdash \psi(x,c_i)$ and so \[ \bigwedge_{i < k} \phi(x,b_{l_i}) \vdash \bigwedge_{i<k} \psi(x,c_{l_i}) \] holds whenever $l_0 < \dots < l_{k-1} < \lambda$ is a $k$-length subsequence of $\lambda$. As the succedent is thin, so is the antecedent. For each such subsequence $s$, let $q_s(y_0,\dots,y_{k-1}) \subset \neg Ix.\bigwedge_{i<k} \phi(x,y_i)$ be a finite subset that witnesses the corresponding conjunction of instances of $\phi(x,y)$ being thin. Per the Erd\H{o}s--Rado theorem, there is an infinite subsequence $(b_i : i < |T|^+)$ (embedded in $(c_i : i < \lambda)$ under, say, $f$) so that the set of $k$-length subsequences is monochromatic under the coloring $(b_{l_0},\dots,b_{l_{k-1}}) \mapsto q_{(f(l_0),\dots,f(l_{k-1}))}$, say with constant value $q_*(y_0,\dots,y_{k-1})$. Passing to this subsequence, we see that $\phi(x,b)$ affords $(b_i : i < |T|^+)$ and $q_*$ for $k$-$I$-dividing over $A$.

For (iii), see \cite{tz}. (iv) is immediate from (iii) and the fact that $\acl(A) = \bigcap\left\{ \M \prec \frak{U} : \frak{M} \supset A \right\}$. (v) is immediate from the definitions.
\end{proof}

\begin{lemma}
    \label{I-dividing_and_indiscernibles}
The following are equivalent:
\begin{enumerate}[a.]
\item $\tp(a/Ab)$ does not $I$-divide over $A$.
\item For every $A$-indiscernible sequence $\b{b} \ni b$ there exist $a' \equiv_{Ab} a$ and $\b{b}' \equiv_{Ab} \b{b}$ such that $\b{b}'$ is $Aa'$-indiscernible and $\tp(a'/A\b{b'})$ is $I$-wide.
\item For every $A$-indiscernible sequence $\b{b} \ni b$ there exists $\b{b}' \equiv_{Ab} \b{b}$ such that $\b{b}'$ is $Aa$-indiscernible and $\tp(a/A\b{b}')$ is $I$-wide.
\item For every $A$-indiscernible sequence $\b{b} \ni b$ there exists $a' \equiv_{Ab} a$such that $\b{b}$ is $Aa'$-indiscernible and $\tp(a'/A\b{b})$ is $I$-wide.
\end{enumerate}
\end{lemma}
\begin{proof}
(c)$\Rightarrow$(b) and (d)$\Rightarrow$(b) are trivial. (b)$\Rightarrow$(c) and (b)$\Rightarrow$(d) are straightforward from the definitions: the only thing to note is that an $Ab'$-conjugate of $a'$ is wide over $Ab'$ if $a'$ is wide over $Ab'$. (c)$\Rightarrow$(a) is also immediate from the definitions: just note that a subset of a wide type is wide.

(a)$\Rightarrow$(b) makes use of a modified version of the standard lemma, stated and proved below. Write $p(x,y) = \tp(ab/A)$ and let $\b{b} = (b_i : i \in J)$ be an $A$-indiscernible sequence containing $b =: b_{i_0}$. Then, by assumption, $\bigcup_{i \in J} p(x,b_i)$ is $I$-wide; so there is a realization $a'$ such that $\tp(a'/A\b{b})$ is wide. Lemma \ref{WSL} delivers an $Aa'$-indiscernible sequence $\b{b}'' = (b_i'' : i \in J)$ which realizes $\EM(\b{b}/Aa')$ and is such that $\tp(a'/A\b{b}'')$ is wide. In particular, we have $\models p(a',b_{i_0}) \cup p(a',b_{i_0}'')$, so there is $\sigma \in \Aut(\frak{U}/Aa')$ sending $b_{i_0}''$ to $b_{i_0}$. Let $\b{b}' = \sigma(\b{b}'')$. Since $\b{b} \ni b$, $\models p(a',b)$, i.e. $a' \equiv_{Ab} a$, holds. Since $\b{b}''$ is $Aa'$-indiscernible, so is $\b{b}' = \sigma(\b{b}'')$. Since $\tp(a'/A\b{b}'')$ is wide, so is $\tp(a'/A\b{b}') = \sigma \cdot \tp(a'/A\b{b}'')$. Finally, since \[ \b{b}'' = \sigma^{-1}(\b{b}') \models \EM(\b{b}/Aa')\supset \EM(\b{b}/A), \] $b \in \b{b}$, $b \in \b{b}'$, and $\b{b}$, $\b{b}'$ are  $A$-indiscernible, we have $\b{b}'\equiv_{Ab} \b{b}$. This completes the proof.
\end{proof}
    
\begin{lemma} \label{WSL}
Let $J,K$ be infinite linear orders, $\b{b}\in \frak{U}_y^J$, and $a \in \frak{U}_x$. Suppose $\tp(a/A\b{b})$ is $I$-wide. Then there is an $Aa$-indiscernible sequence $\b{b}' \in \frak{U}_y^K$ which realizes $\EM(\b{b}/Aa)$ and is such that $\tp(a/A\b{b}')$ is $I$-wide.
\end{lemma}
\begin{proof}
We want to realize $\Gamma = \Gamma_1 \cup \Gamma_2 \cup \Gamma_3$, where
\begin{align*}
\Gamma_1 & = \left\{ \phi(y_{k_1},\dots,y_{k_n}) \leftrightarrow \phi(y_{k_1'},\dots,y_{k_n'}) : n \in \omega, k_1^{(')} < \cdots < k_n^{(')} \in K, \phi(\b{y}) \in \L(Aa) \right\}, \\
\Gamma_2 & = \left\{ \phi(y_1,\dots,y_k) : n \in \omega, k_1 < \cdots < k_n \in K, \phi(\overline{y}) \in \EM(\b{b}/Aa) \right\}, \\
\Gamma_3 & = \left\{ \neg \theta(a,\b{y}) \vee \bigwedge \pi(\b{y}) : \theta(x,\b{y}) \in \L(A), \pi(\b{y}) \subset \neg Ix.\theta(x,\b{y}) \text{ finite} \right\}.
\end{align*}
Let $\Delta \subset \Gamma$ be finite and write
\begin{align*}
\Delta_1 & = \Delta \cap \Gamma_1 = \left\{\phi_{1,i}(\b{y}_{1,i}) \leftrightarrow \phi_{1,i}(\b{y}'_{1,i}) : i < n_1\right\}, \\
\Delta_2 & = \Delta \cap \Gamma_2 = \left\{ \phi_{2,i}(\b{y}_{2,i}) : i < n_2\right\}, \\
\Delta_3 & = \Delta \cap \Gamma_3 = \left\{ \neg\theta_{3,i}(a,\b{y}_{3,i}) \vee \bigwedge \pi_{3,i}(\b{y}_{3,i}) : i < n_3 \right\}.
\end{align*}
By padding the formulas with dummy variables we may assume that each of the formulas written above is in the same free variables $\b{y}$. Let $l = l(\b{y})$. Consider the $2^{n_1}$-coloring of the increasing $l$-length subsequences of $\b{b}$ given by $\b{b}_0 \mapsto \left\{ i < n : \models \phi_{1,i}(\b{b}_0) \right\}$.
By Ramsey's theorem, there is an infinite subsequence $\b{b}_*$ of $\b{b}$ so that the set of increasing $l$-length subsequences of $\b{b}_*$ is monochromatic. Thus $\b{b}_*$ satisfies $\Delta_1$. Since $\b{b}_*$ is a subsequence of $\b{b}$, it satisfies $\Delta_2$; moreover, since $\tp(a/A\b{b})$ is $I$-wide, it satisfies $\Delta_3$, as required.
\end{proof}

\begin{corollary}
(Transitivity) Let $A \subset B$. If $\tp(a/B)$ does not $I$-divide over $A$ and $\tp(c/Ba)$ does not $I$-divide over $Aa$ then $\tp(ac/B)$ does not $I$-divide over $A$.
\end{corollary}
\begin{proof}
Let $b \in B^{<\omega}$ and $\b{b}$ an infinite $A$-indiscernible sequence containing $b$. Since $\tp(a/B)$ does not $I$-divide over $A$, Lemma \ref{I-dividing_and_indiscernibles} delivers an $Aa$-indiscernible $\b{b}' \equiv_{Ab} \b{b}$ such that $\tp(a/A\b{b}')$ is wide. Since $\tp(c/Ba)$ does not $I$-divide over $Aa$, it also delivers an $Aac$-indiscernible sequence $\b{b}'' \equiv_{Aab} \b{b}'$ such that $\tp(c/Aa\b{b}'')$ is wide. But if $\tp(a/A\b{b}')$ is wide and $\b{b}' \equiv_{Aab} \b{b}''$ then $\tp(a/A\b{b}'')$ is also wide. It follows, $ac/A\b{b}''$ is wide with respect to $I^{\otimes l(a)}\otimes I^{\otimes l(c)} = I^{\otimes l(a) + l(c)}$ (Corollary \ref{associativity}). Moreover, $\b{b}'' \equiv_{Ab} \b{b}$ holds. Another application of Proposition \ref{I-dividing_and_indiscernibles} shows that $\tp(ac/Ab)$ does not $I$-divide over $A$, completing the proof.
\end{proof}

\begin{definition}
A partial type $\pi(x)$ \emph{$I$-forks over $A$} if it implies a disjunction $\bigvee_{l < d} \phi_l(x,c_l)$ of formulas $\phi_l(x,c_l) \in \L_x(U)$, each $I$-dividing over $A$. $c$ is \emph{$I$-independent from $B$ over $A$}, denoted by
\[ c \underset{A}{\anch^I} B, \]
if $\tp(c/AB)$ does not $I$-fork over $A$.
\end{definition}

\begin{lemma}
Let $\M\models T$, $A \subset M$, and $p(x) \in S^I(\M)$.
\begin{enumerate}[i.]
\item If $\M$ is $|A|^+$-saturated, then $p(x)$ $I$-divides over $A$ if and only if $p(x)$ $I$-forks over $A$.
\item If $\M$ is $\kappa$-saturated and strongly $\kappa$-homogeneous for some $\kappa > |A|$ and $p(x)$ $I$-forks over $A$, then it has at least $\kappa$ many conjugates under the action of $\Aut(\M/A)$.
\item Let $\pi(x)$ be a partial type over $A$. If $\pi(x)$ does not $I$-fork over $A$, it extends to some complete $q(x) \in S(A)$ such that $q(x)$ does not $I$-fork over $A$. In particular, if $q(x)\in S(A)$ does not $I$-fork over $A \subset B$, it extends to some $r(x) \in S(B)$ which does not $I$-fork over $A$.
\end{enumerate}
\end{lemma}
\begin{proof}
All of these are straightforward translations of the proofs of the parallel results for forking---see \cite{tz}. For instance, in (ii), assume $p(x)$ $I$-forks over $A$. By (i), $p(x)$ $I$-divides over $A$; suppose $\phi(x,m)\in p(x)$ affords $(m_i : i < \kappa)$ and $q(y_0,\dots,y_{k-1}) \subset \neg Ix.\bigwedge_{i<k}\phi(x,y_i)$ for $k$-$I$-dividing over $A$. $\kappa$-saturation of $\M$ lets us take $\left\{m_i : i < \kappa\right\} \subset \M_y$, and strong $\kappa$-homogeneity ensures that each $\phi(x,m_i)$ belongs to some $\Aut(\M/A)$-conjugate of $p(x)$. If $|\Aut(\M/A) \cdot p(x)| < \kappa$, then by the pigeonhole principle, some $\Aut(\M/A)$-conjugate $q(x)$ of $p(x)$ would contain infinitely many of the conjugates $\phi(x,m_i)$, $i < \kappa$. As this family has the property that all of its $k$-fold meets are thin, it follows that $q(x)$ is thin. But $p(x)$ is wide and $I$ is invariant, so $q(x)$ is wide too--contradiction.
\end{proof}
\subsection{$I$-simplicity}

Let $\Delta = \left\{\phi_l(x,y) : l < d\right\} \subset \L_x$ be finite.\footnote{Since $\Delta$ is finite, we may assume each $\phi_l(x,y_l)$ has the same parameter variables $y$ by padding.}

\begin{definition}
A \emph{uniformly $\Delta$-$k$-$I$-dividing sequence over $A$ in order type $K$} is a sequence of the form $\vec{\psi} := (\psi_i(x,a_i) : i \in K)$ where 
\begin{enumerate}
\item[i.] $K$ is a linear order and 
\item[ii.] for each $i$ there is a (unique) $l = l(i)$ which satisfies $\models \forall xy\left(\psi_i(x,y) \leftrightarrow \phi_{l(i)}(x,y)\right)$,
\end{enumerate}
such that the following conditions are met:
\begin{enumerate}
\item[iii.] $\left\{\psi_i(x,a_i) : i \in K\right\}$ is wide.
\item[iv.] For each $l < d$ there exists a finite $q_l(y_0,\dots,y_{k-1}) \subset \neg Ix.\bigwedge_{j < k} \phi_l(x,y_j)$ so that for every $i \in K$,  $q_{l(i)}$ witnesses $k$-$I$-dividing for $\psi_i(x,a_i)$ over $Aa_{<i}$.
\end{enumerate}
In this case, we say $\vec{\psi}$ \emph{affords} $Q = \left\{q_l(y_0,\dots,y_{k-1}) : l < d\right\}$ \emph{for uniform $\Delta$-$k$-$I$-dividing over $A$}. 
\end{definition}

\begin{observation}
\begin{enumerate}[i.]
\item If $\vec{\psi}$ affords $Q$ for uniform $\Delta$-$k$-$I$-dividing over $A$, then for any subsequence $\vec{\theta}$ and $B \subset A$, $\vec{\theta}$ affords $Q$ for uniform 
$\Delta$-$k$-$I$-dividing over $B$.
\item The following conditions are all type-definable over $A$ in the parameter variables $y_i$, $i \in K$:
\begin{enumerate}[a.]
\item $\left\{\phi(x,y_i) : i \in K\right\}$ is wide.
\item $\phi(x,y_i)$ affords $q$ and $\b{z} = (z_s : s \in \omega)$ for $k$-$I$-dividing over $A \cup \left\{y_j : j < i\right\}$ if, and only if, 
\begin{align*}
\models & \left\{\theta(y_i,y_{j_0},\dots,y_{j_{n-1}}) \leftrightarrow \theta(z_s,y_{j_0},\dots,y_{j_{n-1}}) : \theta \in \L(A), (j_0,\dots,j_{n-1}) \in \left\{j \in K : j < i\right\}^{<\omega}, s \in \omega\right\} \cup \\ & \left\{\neg\bigwedge q(z_{s_0},\dots,z_{s_{k-1}}) : (s_0,\dots,s_{k-1}) \in \omega^k\right\};
\end{align*} 
so this condition is type-definable in the $y_i$ and $z_s$. The projection of a type-definable set is still type-definable in a (sufficiently) saturated model. Therefore, the condition ``$q$ witnesses $\phi(x,y_i)$ uniformly $k$-$I$-dividing over $A$'' is type-definable in the $z_i$.
\end{enumerate}
It follows, there is a type $\Gamma(y_i : i \in K)$ with the property that $(\phi_i(x,b_i) : i \in K)$ affords $q$ for uniform $\phi$-$k$-$I$-dividing over $A$ if, and only if, $\models \Gamma(b_i : i \in K)$. By compactness, there is an infinite uniformly $\phi$-$k$-$I$-dividing sequence over $A$ if, and only if, there is a finite $q \subset \neg Ix.\bigwedge_{j<k}\phi(x,y_j)$ so that for every $N \in \omega$, there is a sequence over $A$ of length $N$ which affords $q$ for uniform $\phi$-$k$-$I$-dividing.
\end{enumerate}
\end{observation}

\begin{definition}
$T$ is \emph{$I$-simple} if for every $x$, $\phi(x,y) \in \L_x$, $k \in \omega$, and $A$ there are no infinite uniformly $\phi$-$k$-$I$-dividing sequences over $A$.
\end{definition}

\begin{proposition}
    \label{simple prop}
The following are equivalent:
\begin{enumerate}[a.]
\item For every finite $\Delta = \left\{\phi_l(x,y) : l < d\right\}$, $k \in \omega$, finite $q_l(\overline{y}) \subset \neg Ix.\bigwedge_{j < k} \phi_l(x,y_j)$ for $l < d$, and $A$ there is a finite bound on the lengths of sequences which afford $\left\{q_l : l < d\right\}$ for uniform $\Delta$-$k$-$I$-dividing over $A$.
\item For every $\phi(x,y)$, $k \in \omega$, finite $q(\overline{y}) \subset \neg Ix.\bigwedge_{j<k}\phi(x,y_j)$, and $A$ there is a finite bound on the lengths of sequences which afford $q(\overline{y})$ for uniform $\phi$-$k$-$I$-dividing over $A$.
\item For every $\phi(x,y)$, $k \in \omega$, finite $q(\overline{y}) \subset \neg Ix.\bigwedge_{j<k}\phi(x,y_j)$, and $A$ there are no sequences of length $\omega$ which afford $q(\overline{y})$ for uniform $\phi$-$k$-$I$-dividing over $A$.
\item For every $\phi(x,y)$, $k \in \omega$, finite $q(\overline{y}) \subset \neg Ix.\bigwedge_{j<k}\phi(x,y_j)$, and $A$ there are no sequences of length $\kappa$ which afford $q(\overline{y})$ for uniform $\phi$-$k$-$I$-dividing over $A$, for every $\kappa \geq \aleph_0$.
For every $\phi(x,y)$, $k \in \omega$, finite $q(\overline{y}) \subset \neg Ix.\bigwedge_{j<k}\phi(x,y_j)$, and $A$ there are no sequences of length $\kappa$ which afford $q(\overline{y})$ for uniform $\phi$-$k$-$I$-dividing over $A$, for some $\kappa \geq \aleph_0$.
\end{enumerate}
\end{proposition}
\begin{proof}
(a)$\Rightarrow$(b), (b)$\Rightarrow$(c), (c)$\Rightarrow$(d), and (d)$\Rightarrow$(e) are trivial. For (e)$\Rightarrow$(b), assume that for some $\phi(x,y)$, $k \in \omega$, finite $q(\overline{y})\subset\neg Ix.\bigwedge_{j<k}\phi(x,y_j)$, and $A$ there are arbitrarily long finite sequences which afford $q(\overline{y})$ for uniform $\phi$-$k$-$I$-dividing over $A$. Fix an arbitrary linear order $K$; we want to show that the condition ``$(\phi(x,z_i) : i \in K)$ affords $q(\overline{y})$ for uniform $\phi$-$k$-$I$-dividing over $A$'' is type-definable by some $\Gamma(z_i : i \in K)$. For then, our assumption guarantees that $\Gamma(z_i : i \in K)$ is satisfiable for any linear order $K$, i.e. there are arbitrarily long infinite uniformly $\phi$-$k$-$I$-dividing sequences. To this end, we proceed in stages:

\begin{enumerate}[(i)]    
\item $\left\{\phi(x,z_i) : i \in k\right\}$ is $I$-wide if and only if $\models \Gamma_0(z_i : i \in K)$ where \[ \Gamma_0(z_i : i \in K) = \bigcup \left\{\neg Ix.\bigwedge_{i \in F} \phi(x,z_i) : F \subset K \text{ finite} \right\}. \]

\item $\phi(x,z_i)$ affords $q(\overline{y})$ and $(w_s : s \in \omega)$ for $k$-dividing over $A \cup \left\{z_j : j < i\right\}$ if and only if $\models \alpha_i((z_j : j \leq i)^\frown (w_s : s \in \omega))$ where
\begin{align*}
\alpha_i((z_j : j \leq i)^\frown (w_s : s \in \omega)) = & \left\{ \neg\bigwedge q(w_{\vec{s}}) : \vec{s}\in \omega^k\right\} \cup \\
& \left\{\psi(z_i,z_{\vec{j}})) \leftrightarrow \psi(w_s,z_{\vec{j}}) : \psi \in \L(A), \vec{j} \in \left\{j \in K : j < i\right\}^{<\omega}, s \in \omega\right\}.
\end{align*}

\item $\phi(x,z_i)$ affords $q(\overline{y})$ for $k$-$I$-dividing over $A\cup\left\{z_j : j < i\right\}$ if and only if $\models \beta_i(z_j : j \in K)$ where $\beta_i(z_j:j \in K) = \exists\overline{w}\left(\alpha_i((z_j : j \in K)^\frown \overline{w})\right)$ is the projection onto the $(z_j : j \in K)$, which we recall is type-definable over $A$.

\item Finally, $(\phi(x,z_i) : i \in K)$ affords $q(\overline{y})$ for uniform $\phi$-$k$-$I$-dividing over $A$ if and only if $\models \Gamma(z_i : i \in K)$, where $\Gamma = \bigcup\left\{\beta_i : i \in K\right\}\cup\Gamma_0$.
\end{enumerate}

It remains to prove (b)$\Rightarrow$(a). Assume there are a finite $\Delta = \left\{\phi_l(x,y) : l < d\right\}$, $k \in \omega$,  $Q = \left\{q_l(\overline{y}) : l < d\right\}$ with each $q_l(\overline{y})\subset\neg Ix.\bigwedge_{j<k}\phi(x,y_j)$ finite, and $A$ such that for every $N \in \omega$ there is a sequence of length $N$ which affords $Q$ for uniform $\Delta$-$k$-$I$-dividing over $A$. Say $\phi_l(x,y)$ is \emph{persistent up to $N$} if there is a sequence $(\psi_i(x,a_i) : i < M)$ which affords $Q$ for uniform $\Delta$-$k$-$I$-dividing over $A$ and is such that
\[ |\left\{ i < M : \models \forall x\left(\psi_i(x,y)\leftrightarrow\phi_l(x,y)\right)\right\}|\geq N. \]
Since $\Delta$ is finite and we assumed there are arbitrarily long finite uniformly $\Delta$-$k$-$I$-dividing sequences affording $Q$, there is some $l_* < d$ such that $\phi_{l_*}(x,y)$ is persistent up to $N$ for every $N \in \omega$.\footnote{Clearly persistence up to $N$ implies persistence up to $N'$ for all $N'\leq N$.}
Passing to the subsequences of instances of $\phi_{l_*}(x,y)$, we see that there are arbitrarily long finite sequences which afford $q_{l_*}$ for uniform $\phi_{l_*}$-$k$-$I$-dividing over $A$. This completes the proof.
\end{proof}

\begin{definition}
$T$ satisfies \emph{wide local character} (with respect to $I$) if for every $A \subset \frak{U}$, $p(x) \in S^I(A)$,\footnote{$x$ may be an arbitrary (finite) tuple of free variables.} $p(x)$ does not $I$-divide over some $A_0 \subset A$ with $|A_0|\leq|T|$.
\end{definition}

\begin{theorem}
(Wide local character) The following are equivalent:
\begin{enumerate}[a.]
\item $T$ is $I$-simple.
\item $T$ satisfies wide local character with respect to $I$.
\item For every $\omega$-saturated $\M \models T$ and $p(x) \in S^I(\M)$, $p(x)$ does not $I$-divide over some $A_0 \subset M$ with $|A_0|\leq|T|$.
\end{enumerate}
\end{theorem}
\begin{proof}
For (a)$\Rightarrow$(b), suppose $p(x) \in S^I(A)$ $I$-divides over every $A_0 \subset A$ with $|A_0|\leq|T|$. A simple induction delivers a sequence $(\phi_i(x,a_i) : i < |T|^+)$ such that $\left\{\phi_i(x,a_i) : i \in \omega\right\} \subset p(x)$ and for each $i$, $\phi_i(x,a_i)$ uniformly $k_i$-$I$-divides over $\left\{a_j : j < i\right\}$, affording the witness $q_i(y_0,\dots,y_{k_i-1})$. Consider the coloring $\chi : \left\{\phi_i(x,a_i) : i < |T|^+\right\} \to \omega \times \L \times \L$ given by $\chi(\phi_i(x,a_i)) := (k_i,\phi_i(x,y_i),q_i(y_0,\dots,y_{k_i-1}))$. The color set has cardinality $|T|$, so the pigeonhole principle delivers a monochromatic subsequence $(\phi(x,b_i) : i < |T|^+)$ of length $|T|^+$, say with constant value $(k,\phi(x,y),q(y_0,\dots,y_{k-1}))$. By definition, this subsequence forms an infinite uniformly $\phi$-$k$-$I$-dividing sequence over $\emptyset$.

(b)$\Rightarrow$(c) is trivial, so it remains to prove (c)$\Rightarrow$(a).

Suppose $T$ is not $I$-simple, so that for some $\phi(x,y)$, $k \in \omega$, and finite $q(\overline{y})\subset \neg Ix.\bigwedge_{j < k}\phi(x,y_j)$ there is a sequence $(\phi(x,a_i) : i < |T|^+)$ which affords $q(\overline{y})$ for uniform $\phi$-$k$-$I$-dividing over $\emptyset$. By induction, we construct $\omega$-saturated models $M^i_j$, $j \leq i < |T|^+$ such that
\begin{enumerate}[i.]
\item for each $i$, $M^i_j \prec M^i_l$ whenever $j < l \leq i$,
\item $\left\{a_l : l < j\right\} \subset M^i_j$ whenever $j \leq i$, and
\item $\phi(x,a_j)$ affords $q(\overline{y})$ for $k$-$I$-dividing over $M^i_j$ whenever $j \leq i$.
\end{enumerate}
The induction starts by taking $M^0_0$ to be an $\Aut(\frak{U})$-conjugate of $M$ over which $\phi(x,a_0)$ affords $q(\overline{y})$ for $k$-$I$-dividing, per Proposition \ref{I-dividing basics}.iii. For the successor step $i \Rightarrow i + 1$, assume we are given an elementary chain $M^i_j, j \leq i$ of $\omega$-saturated models so that each $M^i_j \supset \left\{a_k : k < j\right\}$ and $\phi(x,a_j)$ affords $q(\overline{y})$ for $I$-dividing over $M^i_j$. Let $M^i_{i+1} \prec \frak{U}$ be an arbitrary $\omega$-saturated elementary extension of $M^i_i \cup \left\{a_j : j < i+1\right\}$. By Proposition \ref{I-dividing basics}.iii, since $\phi(x,a_{i+1})$ affords $q(\overline{y})$ for $k$-$I$-dividing over $\left\{a_j : j < i+1\right\}$, there is an automorphism $\sigma \in \Aut(\frak{U}/\left\{a_j : j < i+1\right\})$ such that $M^{i+1}_{i+1} := \sigma(M^i_{i+1})$ has the property that $\phi(x,a_{i+1})$ affords $q(\overline{y})$ for $k$-$I$-dividing over $M^{i+1}_{i+1}$,. Define, for $j \leq i$, $M^{i+1}_j := \sigma(M^i_j)$. Since $\phi(x,a_j)$ affords $q(\overline{y})$ for $k$-$I$-dividing over $M^i_j$, the same is true for $\phi(x,\sigma(a_j))$ over $\sigma(M^i_j) = M^{i+1}_j$. Clearly $M^{i+1}_j$, $j \leq i+1$ forms an elementary chain of $\omega$-saturated models, and since $\sigma$ fixes $\left\{a_j : j < i+1\right\}$ pointwise, each model $M^{i+1}_j$ contains $\left\{a_k : k < j\right\}$.

For the limit step, assume for induction we are given models $M^\alpha_\beta, \beta \leq \alpha < \delta$ with the desired properties. We want to exhibit an elementary chain $M^{\delta}_{\beta}, \beta \leq \delta$ so that $\left\{a_{<\beta}\right\} \subset M^\delta_\beta$ and $\phi(x,a_\beta)$ affords $q(\overline{y})$ for uniform $\phi$-$k$-$I$-dividing over $\emptyset$. Let $p^\alpha_\beta(z_\beta) = \tp(M^\alpha_\beta/a_{<\beta})$, so that $\alpha \subset \alpha'$ and $\beta \leq \beta'$ implies $p^\alpha_\beta(z_\beta) \subset p^{\alpha'}_{\beta'}(z_{\beta'})$. Let $p_\delta(z_\delta) = \bigcup_{\alpha < \delta} p^\delta_\alpha(z_\alpha)$. By compactness, this is a type; let $N^\delta_\delta$ be an arbitrary realization. The substructures of $N^\delta_\delta$ corresponding to the variable restriction from $z_\delta$ to $z_\alpha$ are denoted by $N^\delta_\alpha$. By inductive hypothesis and definition of the $N^\delta_\alpha$, we have $M^{\alpha+1}_\alpha \models p^{\alpha+1}_\alpha(z_\alpha)$, $N^\delta_\alpha \models p^\delta_\alpha(z_\alpha)$, and $p^{\alpha+1}_\alpha(z_\alpha) \subset p^\delta_\alpha(z_\alpha)$, so in particular we have $N^\delta_\alpha \equiv_{a_{<\alpha + 1}} M^{\alpha+1}_\alpha$ whence each $N^{\delta}_\alpha$ contains $a_{<\alpha}$. Moreover, since $N^{\delta}_\alpha \equiv_{a_{<\alpha + 1}} M^{\alpha+1}_\alpha$ (not just over $a_{<\alpha}$) and $\phi(x,a_\alpha)$ affords $q(\overline{y})$ for $k$-$I$-dividing over $M^{\alpha+1}_\alpha$, $\phi(x,a_\alpha)$ also affords $q(\overline{y})$ for $k$-$I$-dividing over $N^\delta_\alpha$. Similarly we see each $N^{\delta}_\alpha$, $\alpha < \delta$ is $\omega$-saturated; for $N^\delta_\delta$, there is no harm passing to an $\omega$-saturated elementary extension.

Now $\phi(x,a_\delta)$ affords $q(\overline{y})$ for $k$-$I$-dividing over $a_{<\delta} \subset N^\delta_\delta$, so there is an $\Aut(\frak{U}/a_{<\delta})$-conjugate $M^\delta_\delta := \sigma N^\delta_\delta$ over which $\phi(x,a_\delta)$ affords $q(\overline{y})$ for $k$-$I$-dividing. Let $M^\delta_\alpha :=\sigma N^\delta_\alpha$. Then each $M^\delta_\alpha$ is $\omega$-saturated, $\beta \leq \alpha \leq \delta$ implies $M^\delta_\beta \prec M^\delta_\alpha$, and $\beta \leq \delta$ implies $a_{<\beta} \subset M^\delta_\beta \equiv_{a_{<\delta}} N^\delta_\beta $. Moreover, for $\beta < \delta$, $\phi(x,\sigma a_\beta) = \phi(x,a_\beta)$ affords $q(\overline{y})$ for $\phi$-$k$-$I$-dividing over $\sigma N^\delta_\beta = M^\delta_\beta$, and we have already established that $\phi(x,a_\delta)$ affords $q(\overline{y})$ for $\phi$-$k$-$I$-dividing over $M^\delta_\beta$. Hence the induction continues.

Let $M^* := M^{|T|^+}_{|T|^+}$. Our sequence $(\phi(x,a_\alpha):\alpha < |T|^+)$ affords $q(\overline{y})$ for uniform $\phi$-$k$-$I$-dividing over $\emptyset$. In particular, $\left\{\phi(x,a_\alpha) : \alpha < |T|^+\right\}$ is wide; let $r(x) \in S^I(M^*)$ be any wide completion to $M^* \supset a_{<|T|^+}$. Fix $A \subset M$ with $|A| \leq |T|$. As $|T|^+$ is regular, there is some $\alpha < |T|^+$ such that $A \subset M^{|T|^+}_\alpha$. By construction, $\phi(x,a_\alpha) \in r(x)$ affords $q(\overline{y})$ for $k$-$I$-dividing over $M^{|T|^+}_\alpha$. Therefore $r(x)$ $I$-divides over $A$, and it thereby witnesses the failure of wide local character over $\omega$-saturated models, as desired.
\end{proof}

\begin{corollary} \label{existence}
Let $T$ be $I$-simple and $p(x) \in S^I(A)$. Then $p(x)$ does not $I$-fork over $A$.
\end{corollary}
\begin{proof}
Suppose $p(x)$ $I$-forks over $A$, so that $p(x) \vdash \bigvee_{l<d} \phi_l(x,b)$ where each $\phi_l(x,b)$ $k$-$I$-divides over $A$, witnessed by some $q_l(y_0,\dots,y_k) \subset \neg Ix.\bigwedge_{i<k}\phi(x,y_i)$. Put $\Delta = \left\{\phi_l(x,y) : l < d\right\}$. We will show by induction on $n \in \omega$ there is a uniformly $\Delta$-$k$-$I$-dividing sequence over $A$ of length $n$ which is wide with $p(x)$ and affords the witnesses $\left\{q_l(y_0,\dots,y_k) : l < d\right\}$. By Proposition \ref{simple prop}, this implies $T$ is not $I$-simple.

Assume for induction that $(\psi_i(x,a_i) : i < n)$ affords the witnesses $\left\{q_l(y_0,\dots,y_k) : l < d\right\}$ for uniform $\Delta$-$k$-$I$-dividing over $A$, such that $p(x) \cup \left\{\psi_i(x,a_i) : i < n\right\}$ is wide. By Proposition \ref{I-dividing basics}.iii, there exists $b' \equiv_A b$ such that $(\psi_i(x,a_i) : i < n)$ affords the witnesses $\left\{q_l(y_0,\dots,y_k) : l < k\right\}$ for uniform $\Delta$-$k$-$I$-dividing over $Ab'$. Since $p(x) \cup \left\{\psi_i(x,a_i) : i < n\right\}$ is wide and $p(x) \vdash \bigvee_{l<d} \phi_l(x,b')$, there must be some $l_0 < d$ such that $\phi_{l_0}(x,b')$ is wide with $p(x) \cup \left\{\psi_i(x,a_i) : i < n\right\}$. Now $\phi_{l_0}(x,b)$ affords the witness $q_{l_0}(y_0,\dots,y_{k-1})$ for $k$-$I$-dividing over $A$, so the same conclusion holds for $\phi_{l_0}(x,b')$ because $b' \equiv_A b$. Putting all of this together, we conclude that $(\phi_{l_0}(x,b'),\psi_0(x,a_0),\dots,\psi_{n-1}(x,a_{n-1})$ is wide with $p(x)$ and affords the witnesses $\left\{q_l(y_0,\dots,y_{k-1}) : l < d\right\}$ for uniform $\Delta$-$k$-$I$-dividing over $A$. This completes the proof.
\end{proof}

\begin{corollary}
If $T$ is $I$-stable, then $T$ is $I$-simple.
\end{corollary}
\begin{proof}
Let $T$ be $I$-stable, $\M \models T$ $\omega$-saturated, and $p(x) \in S^I(\M)$. Then $p(x)$ is $M$-definable by Theorem \ref{V-WUFT}, so collecting together choices of parameters in $M$ for $d_p x.\phi(x,y)$ as $\phi(x,y)$ ranges over $\L_x$, we obtain a subset $M_0 \subset M$ with $|M_0|\leq |T|$ such that $p(x)$ is $M_0$-definable, and thus in particular does not $I$-divide over $M_0$. So $T$ satisfies wide local character over $\omega$-saturated models, which implies $T$ is $I$-simple in view of the preceding theorem.
\end{proof}
\subsection{Independence in $I$-simple theories}

\begin{definition}
Let $\b{a} = (a_i : i \in J)$ be a sequence in $\frak{U}$, $l(a_i) = l$. $\b{a}$ is \emph{$I$-independent from $B$ over $A$} if for each $i \in J$, 
\[ a_i \underset{A}{\anch^I} Ba_{<i}, \]
i.e. $\tp(a_i/ABa_{<i})$ does not $I$-fork over $A$. $\b{a}$ is an \emph{$I$-Morley sequence with respect to $B$ over $A$} if it is $I$-independent from $B$ over $A$ and $A$-indiscernible. If $B = \emptyset$, we omit ``with respect to $B$.''
\end{definition}

Note that $\b{a}$ is $I$-independent from $B$ over $A$ if and only if $\b{a}$ is $I$-independent from $AB$ over $A$. 

\begin{proposition} \label{existence of I-MS}
If $p(x) \in S^I(B)$ does not $I$-fork over $A$, there is an infinite sequence in $p(x)$ which is $I$-Morley sequence with respect to $B$ over $A$ and $B$-indiscernible. If $T$ is $I$-simple, then for every $p(x) \in S^I(A)$, there is an infinite $I$-Morley sequence over $A$ in $p(x)$.
\end{proposition}
\begin{proof}
Straightforward translation of the proof of the parallel result for Morley sequences in simple theories; see \cite{tz}.
\end{proof}

\begin{lemma}
\label{WKL}
Let $T$ be $I$-simple. Let $\pi(x,y)$ be a partial type over $A$ and $\b{b} = (b_i : i \in \omega)$ an infinite $I$-Morley sequence over $A$. If $\bigcup_{i \in \omega} \pi(x,b_i)$ is $I$-wide, then $\pi(x,b_0)$ does not $I$-divide over $A$. 
\end{lemma}
\begin{proof}
Fix a linear order $K$; by the standard lemma, there is an $A$-indiscernible sequence $\b{c}$ in order type $K$ such that $\b{c} \models \EM(\b{b}/A)$. Note that since $\b{b}$ is $A$-indiscernible, for each $n$, the set $\EM_n(\b{b}/A)$ of $\L(A)$-formulas in the free variables $(y_i : i < n)$ satisfied by all increasing $n$-length subsequences of $\b{b}$ is a complete $n\cdot l(y)$-type. In particular, as $\models \neg Ix.\bigwedge_{j \in J_0} \pi(x,b_j)$ holds whenever $J_0 \subset J$ is finite, so too $\models \neg Ix.\wedge_{k \in K_0} \pi(x,c_k)$ whenever $K_0 \subset K$  is finite. Thus $\bigcup_{k \in K} \pi(x,c_k)$ is $I$-wide. We further claim that for each $k \in K$, 
\[ c_k \underset{A}{\anch^I} c_{<k}. \]
Indeed, assume $\tp(c_k/Ac_{<k})$ $I$-forks over $A$. Then $\tp(c_k/Ac')$ $I$-forks over $A$ for some finite $c' \subset c_{<k}$; let $l = l(c')$. Since $c'c_k$ and $b_{<l}b_l$ both satisfy \[ \EM_{l+1}(\b{b}/A) \in S_{(l+1)l(y)}(A), \]
$\tp(b_l/Ab_{<l})$ does not $I$-fork over $A$, contradicting the assumption that $\b{b}$ is $I$-independent over $A$.

Now let $K = (|T|^+,<^*)$ where $i <^* j \Longleftrightarrow j \in i$ for $i,j \in |T|^+$. The argument just given delivers a sequence $\b{c}$ which is $I$-Morley with respect to $b$ over $A$ in order type $K$ such that $\bigcup_{k\in K}\pi(x,c_k)$ is $I$-wide. So there is a realization $d \models \bigcup_{k\in K} \pi(x,c_k)$ such that $\tp(d/A\b{c})$ is wide. By wide local character, $\tp(d/A\b{c})$ does not $I$-divide over $A\left\{c : i \in i_0\right\} = A\left\{c_i : i_0 <^* i\right\}$ for some $i_0 \in |T|^+$. In particular, $d$ does not $I$-divide with $c_{i_0}$ over $A\left\{c_i : i_0 <^* i\right\}$. On the other hand, using the fact that $\b{c}$ is $I$-independent over $A$ along with an application of transitivity of non-$I$-dividing, we see $(c_i : i_0 <^* i)$ does not $I$-divide with $c_{i_0}$ over $A$. Therefore, by another application of transitivity, it follows that $(c_i : i_0 <^* i)^\frown d$ does not $I$-divide with $c_{i_0}$ over $A$. But $d \models \pi(x,c_{i_0})$, so $\pi(x,c_{i_0})$ does not $I$-divide over $A$. As $b_0 \equiv_A c_{i_0}$, we conclude that $\pi(x,b_0)$ does not $I$-divide over $A$, as desired.
\end{proof}

\begin{corollary} \label{I-forking=I-dividing}
Let $T$ be $I$-simple. Then $\pi(x,b)$ $I$-forks over $A$ if and only if $\pi(x,b)$ $I$-divides over $A$.
\end{corollary}
\begin{proof}
Suppose $\pi(x,b)$ does not $I$-divide over $A$. Assume 
\[\pi(x,b) \vdash \bigvee_{l < d} \phi_l(x,b) =: \psi(x,b)\]
(so $\psi(x,b)$ does not $I$-divide over $A$ by definition). We want to show $\phi_l(x,b)$ does not $I$-divide over $A$ for some $l < d$. Let $\b{b} = (b_i : i \in \omega)$ be an infinite $I$-Morley sequence over $A$ in $\tp(b/A)$. Since $\psi(x,b)$ does not $I$-divide over $A$ and $\b{b}$ is $A$-indiscernible, the type $\left\{\psi(x,b_i) : i \in \omega\right\}$ is wide. In view of Lemma \ref{WKL} and because $\b{b}$ is $I$-Morley over $A$, it suffices to show that there exist $l < d$ and an infinite $K \subset \omega$ such that $\left\{\phi_l(x,b_i) : i \in I\right\}$ is wide. Now 
\[ \left\{\bigvee_{l < d} \phi(x,b_i) : i \in \omega\right\} \]
is wide, so by saturation there is a point $a$ whose complete type $\tp(a/\b{b})$ over $\b{b}$ is wide. Consider the coloring $\omega \to d$ which assigns to $i$ the least $l < d$ such that $\models \phi_l(a,b_i)$. The pigeonhole principle delivers some $l_0 < d$ and an infinite $K \subset \omega$ such that $\models \phi_{l_0}(a,b_i)$ holds for all $i \in K$. Moreover, since $\tp(a/b{b})$ is wide, so is $\left\{\phi_{l_0}(x,b_i) : i \in K\right\}$, as required.
\end{proof}

As of writing this article, the author could not determine whether the assumption that $I$ is $\bigvee$-definable alone implies symmetry of non-$I$-forking independence. Accordingly, we conclude with a

\begin{problem}
Let $I$ be a schema of $\bigvee$-definable ideals for a complete theory $T$ and let $\frak{U}$ be a sufficiently saturated and strongly homogeneous model of $T$. Find necessary and sufficient conditions on $I$ which characterize when ${\anch^I}$-independence is symmetric: for all $a$, $b \in U^{<\omega}$ and small $C \subset U$, if $\tp(a/Cb)$ and $\tp(b/Ca)$ are $I$-wide, then 
\[a {\anch^I_C} b \hspace{0.25cm} \Longleftrightarrow \hspace{0.25cm} b {\anch^I_C} a.\]
\end{problem}

\printbibliography

@article {bl,
    AUTHOR = {Baldwin, J. T. and Lachlan, A. H.},
     TITLE = {On strongly minimal sets},
   JOURNAL = {J. Symbolic Logic},
  FJOURNAL = {The Journal of Symbolic Logic},
    VOLUME = {36},
      YEAR = {1971},
     PAGES = {79--96},
      ISSN = {0022-4812,1943-5886},
   MRCLASS = {02.50},
  MRNUMBER = {286642},
MRREVIEWER = {Walter\ Taylor},
       DOI = {10.2307/2271517},
}

@article {cs,
    AUTHOR = {Chernikov, Artem and Starchenko, Sergei},
     TITLE = {A note on the {Erd\H{o}s}-{Hajnal} property for stable graphs},
   JOURNAL = {Proc. Amer. Math. Soc.},
  FJOURNAL = {Proceedings of the American Mathematical Society},
    VOLUME = {146},
      YEAR = {2018},
    NUMBER = {2},
     PAGES = {785--790},
      ISSN = {0002-9939,1088-6826},
   MRCLASS = {03C45 (05C35 05C69)},
  MRNUMBER = {3731711},
MRREVIEWER = {Michael\ Kompatscher},
       DOI = {10.1090/proc/13626},
}

@misc{chernikov–towsner,
      title={Perfect stable regularity lemma and slice-wise stable hypergraphs}, 
      author={Artem Chernikov and Henry Towsner},
      year={2024},
      eprint={2402.07870},
      archivePrefix={arXiv},
      primaryClass={math.CO},
}

@misc{cm,
      title={Countable {Ramsey}}, 
      author={Leonardo N. Coregliano and Maryanthe Malliaris},
      year={2022},
      eprint={2203.10396},
      archivePrefix={arXiv},
      primaryClass={math.CO},
}

@misc{giron,
      title={On the regularity of almost stable relations}, 
      author={Marcos Girón},
      year={2025},
      eprint={2508.00511},
      archivePrefix={arXiv},
      primaryClass={math.LO},
}

@book {halmos,
    AUTHOR = {Halmos, Paul R.},
     TITLE = {Lectures on {B}oolean algebras},
    SERIES = {Van Nostrand Mathematical Studies},
    VOLUME = {No. 1},
 PUBLISHER = {D. Van Nostrand Co., Inc., Princeton, NJ},
      YEAR = {1963},
     PAGES = {v+147},
   MRCLASS = {06.60},
  MRNUMBER = {167440},
MRREVIEWER = {R.\ Sikorski},
}

@book {hodges,
    AUTHOR = {Hodges, Wilfrid},
     TITLE = {Model theory},
    SERIES = {Encyclopedia of Mathematics and its Applications},
    VOLUME = {42},
 PUBLISHER = {Cambridge University Press, Cambridge},
      YEAR = {1993},
     PAGES = {xiv+772},
      ISBN = {0-521-30442-3},
   MRCLASS = {03-01 (03-02 03Cxx)},
  MRNUMBER = {1221741},
MRREVIEWER = {J.\ M.\ Plotkin},
       DOI = {10.1017/CBO9780511551574},
}

@article {approxeq,
    AUTHOR = {Hrushovski, Ehud},
     TITLE = {Approximate Equivalence Relations},
   JOURNAL = {Model Theory},
  FJOURNAL = {Model Theory},
    VOLUME = {3},
      YEAR = {2024},
    NUMBER = {2},
     PAGES = {317--416},
      ISSN = {2832-9058,2832-904X},
   MRCLASS = {03C07 (03B48 03C66)},
  MRNUMBER = {4776304},
MRREVIEWER = {Sergey\ V.\ Sudoplatov},
       DOI = {10.2140/mt.2024.3.317},
}

@article {approxsg,
    AUTHOR = {Hrushovski, Ehud},
     TITLE = {Stable group theory and approximate subgroups},
   JOURNAL = {J. Amer. Math. Soc.},
  FJOURNAL = {Journal of the American Mathematical Society},
    VOLUME = {25},
      YEAR = {2012},
    NUMBER = {1},
     PAGES = {189--243},
      ISSN = {0894-0347,1088-6834},
   MRCLASS = {03C45 (11P70)},
  MRNUMBER = {2833482},
MRREVIEWER = {Katrin\ U.\ Tent},
       DOI = {10.1090/S0894-0347-2011-00708-X},
}

@article {kim,
    AUTHOR = {Kim, Byunghan},
     TITLE = {Forking in simple unstable theories},
   JOURNAL = {J. London Math. Soc. (2)},
  FJOURNAL = {Journal of the London Mathematical Society. Second Series},
    VOLUME = {57},
      YEAR = {1998},
    NUMBER = {2},
     PAGES = {257--267},
      ISSN = {0024-6107,1469-7750},
   MRCLASS = {03C45},
  MRNUMBER = {1644264},
MRREVIEWER = {John\ T.\ Baldwin},
       DOI = {10.1112/S0024610798005985},
}

@article {rlfsg,
    AUTHOR = {Malliaris, M. and Shelah, S.},
     TITLE = {Regularity lemmas for stable graphs},
   JOURNAL = {Trans. Amer. Math. Soc.},
  FJOURNAL = {Transactions of the American Mathematical Society},
    VOLUME = {366},
      YEAR = {2014},
    NUMBER = {3},
     PAGES = {1551--1585},
      ISSN = {0002-9947,1088-6850},
   MRCLASS = {05C75 (03C45 03C98)},
  MRNUMBER = {3145742},
MRREVIEWER = {Konrad\ P.\ Pi\'oro},
       DOI = {10.1090/S0002-9947-2013-05820-5},
}

@article {morley,
    AUTHOR = {Morley, Michael},
     TITLE = {Categoricity in power},
   JOURNAL = {Trans. Amer. Math. Soc.},
  FJOURNAL = {Transactions of the American Mathematical Society},
    VOLUME = {114},
      YEAR = {1965},
     PAGES = {514--538},
      ISSN = {0002-9947,1088-6850},
   MRCLASS = {02.52},
  MRNUMBER = {175782},
MRREVIEWER = {H.\ B.\ Curry},
       DOI = {10.2307/1994188},
}

@article {mp,
    AUTHOR = {Malliaris, Maryanthe and Pillay, Anand},
     TITLE = {The stable regularity lemma revisited},
   JOURNAL = {Proc. Amer. Math. Soc.},
  FJOURNAL = {Proceedings of the American Mathematical Society},
    VOLUME = {144},
      YEAR = {2016},
    NUMBER = {4},
     PAGES = {1761--1765},
      ISSN = {0002-9939,1088-6826},
   MRCLASS = {03C45 (03C98 05C75)},
  MRNUMBER = {3451251},
MRREVIEWER = {Yun\ Lu},
       DOI = {10.1090/proc/12870},
}

@misc{mppw,
      title={Stability, corners, and other 2-dimensional shapes}, 
      author={Amador Martin-Pizarro and Daniel Palacin and Julia Wolf},
      year={2024},
      eprint={2210.14039},
      archivePrefix={arXiv},
      primaryClass={math.LO},
}

@book {ct,
    AUTHOR = {Shelah, S.},
     TITLE = {Classification theory and the number of nonisomorphic models},
    SERIES = {Studies in Logic and the Foundations of Mathematics},
    VOLUME = {92},
   EDITION = {Second},
 PUBLISHER = {North-Holland Publishing Co., Amsterdam},
      YEAR = {1990},
     PAGES = {xxxiv+705},
      ISBN = {0-444-70260-1},
   MRCLASS = {03C45 (03-02)},
  MRNUMBER = {1083551},
MRREVIEWER = {Perry\ Smith},
}

@article {stone,
    AUTHOR = {Stone, M. H.},
     TITLE = {The theory of representations for {B}oolean algebras},
   JOURNAL = {Trans. Amer. Math. Soc.},
  FJOURNAL = {Transactions of the American Mathematical Society},
    VOLUME = {40},
      YEAR = {1936},
    NUMBER = {1},
     PAGES = {37--111},
      ISSN = {0002-9947,1088-6850},
   MRCLASS = {06E20},
  MRNUMBER = {1501865},
       DOI = {10.2307/1989664},
}

@book{simon, 
      place={Cambridge}, 
      series={Lecture Notes in Logic}, 
      title={A Guide to {NIP} Theories}, 
      publisher={Cambridge University Press}, 
      author={Simon, Pierre}, 
      year={2015}, 
      collection={Lecture Notes in Logic}
}

@book {tz,
    AUTHOR = {Tent, Katrin and Ziegler, Martin},
     TITLE = {A course in model theory},
    SERIES = {Lecture Notes in Logic},
    VOLUME = {40},
 PUBLISHER = {Association for Symbolic Logic, La Jolla, CA; Cambridge
              University Press, Cambridge},
      YEAR = {2012},
     PAGES = {x+248},
      ISBN = {978-0-521-76324-0},
   MRCLASS = {03C35 (03-01 03C45)},
  MRNUMBER = {2908005},
MRREVIEWER = {David\ Evans},
       DOI = {10.1017/CBO9781139015417},
}

\end{document}